\documentclass[11pt,a4paper]{article}

\usepackage{mathrsfs}
\usepackage{amssymb}
\usepackage{latexsym}
\usepackage{amsmath}
\usepackage{galois}
\usepackage{amsthm}
\usepackage{enumerate}
\usepackage{color,xcolor}
\usepackage{comment}
\usepackage[colorlinks]{hyperref}
\usepackage{tikz}
\usepackage{subfigure}
\usepackage{bm}

\usepackage{graphicx}
\usepackage{threeparttable}
\usepackage{booktabs}
\usepackage{tabularx}
\usepackage{diagbox}
\usepackage{multicol}
\usepackage{multirow}

\newcommand{\norm}[1]{\left\Vert#1\right\Vert}
\newcommand{\abs}[1]{\left\vert#1\right\vert}
\newcommand{\inp}[1]{\left\langle#1\right\rangle}

\newcommand{\argmin}{\operatornamewithlimits{arg\,min}}
\newcommand{\argmax}{\operatornamewithlimits{arg\,max}}

\def\d{\mathrm{d}}

\def\cl{\mathrm{cl}}
\def\bd{\mathrm{bd}}

\def\prob{\mathrm{Prob}}

\newtheorem{theorem}{Theorem}[section]
\newtheorem{corollary}{Corollary}[section]
\newtheorem{lemma}{Lemma}[section]
\newtheorem{proposition}{Proposition}[section]

\newtheorem{definition}{Definition}[section]
\newtheorem{example}{Example}[section]
\newtheorem{remark}{Remark}[section]

\author{
    Jie Jiang
    \footnotemark[2]
    \and
    Xiaojun Chen
    \footnotemark[3]
    \footnotemark[1]
    }

\title{Optimality conditions for nonsmooth nonconvex-nonconcave min-max problems and generative adversarial networks}

\usepackage[top=1in, bottom=1in, left=1.25in, right=1.25in]{geometry}
\date{}

\begin{document}
\footnotetext[1]{Corresponding Author. This work is supported by The Hong Kong Polytechnic University Post-doctoral fellow scheme and The Hong Kong Grant Council Grant PolyU15300120.}
\footnotetext[2]{CAS AMSS-PolyU Joint Laboratory of Applied Mathematics, The Hong Kong Polytechnic University, Hong Kong, China, jiangjiecq@163.com}
\footnotetext[3]{Department of Applied Mathematics, The Hong Kong Polytechnic University, Hong Kong, China, xiaojun.chen@polyu.edu.hk}
\maketitle

\begin{abstract}
This paper considers a class of nonsmooth nonconvex-nonconcave min-max problems in machine learning and games. We first provide sufficient conditions for the existence of global minimax points and local minimax points.   Next, we establish the first-order and second-order optimality conditions for local minimax points by using directional derivatives. These conditions reduce to smooth min-max problems with Fr{\'e}chet derivatives.  We apply our theoretical results to generative adversarial networks (GANs) in which two neural networks contest with each other in a game. Examples are used to illustrate applications of the new theory for training GANs.
\end{abstract}


\section{Introduction}

Consider the following min-max problem
\begin{equation}
\label{minmax}
\min_{x\in X} \max_{y\in Y}~  f(x,y),
\end{equation}
where $X\subseteq \mathbb{R}^n$ and $Y\subseteq \mathbb{R}^m$ are nonempty, closed and convex sets, $f:\mathbb{R}^{n}\times \mathbb{R}^{m}\to \mathbb{R}$ is a locally Lipschitz continuous function. Define an envelope function
$$\varphi(x):=\max_{y\in Y}f(x,y).$$
In this paper, we assume that $\varphi(x)$ is finite-valued for any $x\in X$.
We say problem \eqref{minmax} is nonconvex-nonconcave if for a fixed $x\in X$, $f(x,\cdot)$ is not concave, and for a fixed $y\in Y$, $f(\cdot,y)$ is not convex.

The min-max problem \eqref{minmax} has many applications in machine learning and games \cite{JC2021pure,Qian2019Robust,ZHWZ2020generalization}, for instance, the popular generative adversarial networks (GANs) in machine learning  \cite{ACB2017wasserstein,C2018generative,G2016nips,G2014generative,M2021statistical}.
Let $D:\mathbb{R}^m\times \mathbb{R}^{s_1} \to (0,1)$ be a parameterized discriminator, $G:\mathbb{R}^n\times \mathbb{R}^{s_2} \to \mathbb{R}^{s_1}$ be a parameterized generator and $\xi_i$ be a $s_i$-valued random vector with probability distribution $P_i$ and support $\Xi_i\subseteq \mathbb{R}^{s_i}$ for each $i=1,2$. Then the plain vanilla GAN model can be formulated as
\begin{equation}
\label{GAN}
\min_{x \in X} \max_{y\in Y}~  \mathbb{E}_{P_{1}}\Big[\log\big(D(y,\xi_1)\big)\Big] +  \mathbb{E}_{P_2}\Big[\log\Big( 1 - D\big(y,G(x,\xi_2)\big)\Big)\Big],
\end{equation}
where $x$ and $y$ are the parameters to control $D$ and $G$ with ranges $X$ and $Y$, respectively. Here $\mathbb{E}_{P_i}[\cdot]$ denotes the expected value  with probability distribution $P_i$ over $\Xi_i$ for $i=1,2$. We assume that the expected values are finite for any fixed $x\in X$ and $y\in Y$.  Since the range of $D$ is $(0,1)$, for any fixed $x$,
 $$\varphi(x)=\max_{y\in Y}~  \mathbb{E}_{P_{1}}\Big[\log\big(D(y,\xi_1)\big)\Big] +  \mathbb{E}_{P_2}\Big[\log\Big( 1 - D\big(y,G(x,\xi_2)\big)\Big)\Big]$$
is real-valued.  The functions $D$ and $G$ are usually defined by deep neural networks (see Section \ref{Sec4} for a specific example).

Since the pioneering work \cite{N1928theorie} by Von Neumann in 1928, convex-concave min-max problems have been investigated extensively, based on the concept of saddle points (see e.g. \cite{CLO2014optimal,N2004prox,ZHWZ2020generalization,ZHZ2019lower} and the references therein).
In the recent years, driving by important applications, nonconvex-nonconcave min-max problems have attracted considerable attention  \cite{JNJ2020local,LLRY2018solving,LTH2019block,RLLY2021weakly}.
However, it is well-known that a nonconvex-nonconcave min-max problem may not have a saddle point.
How to properly define its local optimal points and optimality conditions has been of great concern. In \cite{ADLH2019local,DP2018limit,MRS2020gradient},  the concept of local saddle points was studied, but it is pointed out in \cite{JNJ2020local} that the concept of local saddle points is not suitable
for most applications of min-max optimization in machine learning.
A nonconvex-nonconcave min-max problem may not have a local saddle point (see Example \ref{Eg1} in this paper).
In \cite{JNJ2020local}, the authors argued that a local solution cannot be determined just based on the function value in an arbitrary small neighborhood of a given point. For that reason, they proposed the concept of local minimax points of unconstrained smooth nonconvex-nonconcave min-max problems and studied the first-order and second-order optimality conditions.

The main contributions of this paper can be summarized as follows.
\begin{itemize}

\item We define the first-order and second-order optimality conditions of local minimax points of constrained min-max problem (\ref{minmax}) by using directional derivatives. These conditions reduce to smooth min-max problems with Fr{\'e}chet derivatives. Our optimality conditions extend the work \cite{JNJ2020local} for unconstrained smooth min-max problems to constrained nonsmooth min-max problems. Moreover, we rigorously describe the relationship between saddle points, local saddle points, global minimax points, local minimax points and stationary points defined by these first-order and second-order optimality conditions. The relationship among these points is illustrated by interesting examples and summarized in Figure \ref{fig-3}.

\item We establish new mathematical optimization theory for the GAN model with both smooth and nonsmooth activation functions. In particular, we give new properties of global minimax points, local minimax points and stationary points of problem (\ref{GAN}) under some specific settings. Examples with the sample average approximation approach  show that our results are helpful and efficient for training GANs.
\end{itemize}

The reminder of the paper is organized as follows.
In section \ref{Sec2}, we give some notations and preliminaries. In section \ref{Sec3}, we study the first-order and second-order optimality conditions of nonsmooth and smooth min-max problems, respectively. In section \ref{Sec4}, we apply our results to GANs and use examples to show the effectiveness of our results. Finally, we make some concluding remarks in section \ref{Sec5}.

\section{Notations and preliminaries}\label{Sec2}

In this paper, $\mathbb{N}$ denotes the natural numbers. $\mathbb{R}_+^n$ denotes the nonnegative part of $\mathbb{R}^n$. $\norm{\cdot}$ denotes the Euclidean norm. $\cl(\Omega)$, $\mathrm{int} (\Omega)$ and $\bd (\Omega)$ denote the closure, the interior and the boundary of set $\Omega$, respectively.  $o(|t|)$ denotes the infinitesimal of a higher order than $|t|$ as $t\to 0$. $O(|t|)$ denotes the same order as $|t|$ as $t\to 0$.
$\mathbb{B}(x,r)$ denotes the closed ball centred at $x$ with radius $r>0$.

Let $\Omega\subseteq \mathbb{R}^n$ be a closed and convex set. The tangent cone \cite[Definition 6.1]{RW2009variational} to $\Omega$ at $x\in \Omega$, denoted by $\mathcal{T}_\Omega(x)$, is defined as
\begin{align*}
\mathcal{T}_\Omega(x) &= \left\{ w: \exists~ x^k\overset{\Omega}{\to} x, t^k\downarrow 0 ~ \text{such that}~ \lim_{ k\to \infty } \frac{x^k-x}{t^k} = w \right\}.
\end{align*}

The normal cone \cite[Definition 6.3]{RW2009variational} to $\Omega$ at $x\in \Omega$, denoted by $\mathcal{N}_\Omega(x)$, is
$$\mathcal{N}_\Omega(x):=\{y\in \mathbb{R}^n: \inp{y, \omega-x}\leq 0, \forall \omega \in \Omega \}.$$
It also knows from \cite[Proposition 6.5]{RW2009variational} that $\mathcal{N}_\Omega(x)=\{v: \inp{v,\omega}\leq 0,~\text{for}~ \forall \omega\in \mathcal{T}_\Omega(x) \}$.

\begin{definition}\label{Def1}
We say that $(\hat{x},\hat{y})\in X\times Y$ is a saddle point of problem \eqref{minmax}, if
\begin{equation}
\label{saddp}
f(\hat{x},y) \leq f(\hat{x},\hat{y}) \leq  f(x,\hat{y})
\end{equation}
holds for any $(x,y)\in X\times Y$.
\end{definition}

\begin{definition}\label{Def2}
We say that $(\hat{x},\hat{y})\in X\times Y$ is a local saddle point of problem \eqref{minmax}, if there exists a $\delta>0$ such that, for any $(x,y)\in X\times Y$ satisfying $\norm{x-\hat{x}}\leq \delta$ and $\norm{y-\hat{y}}\leq \delta$, \eqref{saddp} holds.
\end{definition}

In the convex-concave setting, saddle points are usually used to describe the optimality of min-max problems.  However, one significant drawback of considering (local) saddle points of nonconvex-nonconcave problems is that such points might not exist \cite[Proposition 6]{JNJ2020local}. Also, (local) saddle points correspond to simultaneous game, but many applications (such as GANs and adversarial training) correspond to sequential games. In view of this, we consider in what follows global and local minimax points proposed in \cite{JNJ2020local}, which are from the viewpoint of sequential games.

\begin{definition}
\label{Def3}
We say that $(\hat{x},\hat{y})\in X\times Y$ is a global minimax point of problem \eqref{minmax}, if
$$f(\hat{x},y) \leq f(\hat{x},\hat{y}) \leq \max_{y'\in Y}f(x,y')$$
holds for any $(x,y)\in X\times Y$.
\end{definition}

\begin{definition}
\label{Def4}
We say that $(\hat{x},\hat{y})\in X\times Y$ is a local minimax point of problem \eqref{minmax}, if there exist a $\delta_0>0$ and a function $\tau:\mathbb{R}_+\rightarrow \mathbb{R}_+$ satisfying $\tau(\delta)\rightarrow 0$ as $\delta\rightarrow 0$, such that for any $\delta\in (0,\delta_0]$ and any $(x,y)\in X\times Y$ satisfying $\norm{x-\hat{x}}\leq \delta$ and $\norm{y-\hat{y}}\leq \delta$, we have
$$f(\hat{x},y) \leq f(\hat{x},\hat{y}) \leq \max_{y'\in\{y\in Y: \norm{y-\hat{y}}\leq \tau(\delta)\}}f(x,y').$$
\end{definition}

\begin{remark}
It is noteworthy that the function $\tau$ in Definition \ref{Def4} can be further restricted to be monotone or continuous without changing Definition \ref{Def4} \cite[Remark 15]{JNJ2020local}. Hereafter, we always assume that $\tau$ is monotone and continuous.

Global or local minimax points are motivated by many practical applications and the probable nonconvexity-nonconcavity of the min-max problem. Obviously, a saddle point is a global minimax point and a local saddle point is a local minimax point. However, problem (\ref{minmax}) may not have a local saddle point.    The following proposition gives some sufficient conditions for the existence of global (local) minimax points.
Note that the existence of a global (local) minimax point does not imply the existence of a local saddle point.
\end{remark}

\begin{proposition}\label{Prop1}
\begin{enumerate}
\item[(i)] If $\Phi_u:=\{x\in X: \varphi(x) \leq u\}$ is nonempty and bounded for some scalar $u$ and $\{y\in Y: f(x,y) \geq l_x\}$ is bounded for every $x\in \Phi_u$ and some scalar $l_x$, then problem \eqref{minmax} has at least a global minimax point.

\item[(ii)] (\cite[Lemma 16]{JNJ2020local}) $(x^*,y^*)\in X\times Y$ is a local minimax point if and only if $y^*$ is a local maximizer of $f(x^*,\cdot)$ and there exists a $\delta_0>0$ such that $x^*$ is a local minimizer of $\varphi_\delta(x):=\max_{y'\in\{y\in Y:\norm{y-y^*}\leq \delta\}} f(x,y')$ for any $\delta\in(0,\delta_0]$.
\end{enumerate}
\end{proposition}

\begin{proof}
(i) According to the continuity of $f(x,y)$, $\varphi$ is lower semicontinuous. We know from \cite[Theorem 1.9]{RW2009variational} that $\argmin_{x\in X} \varphi(x)\subseteq \Phi_u$ is nonempty and compact. Let $x^*\in \argmin_{x\in X} \varphi(x)$ and consider the set $\argmax_{y\in Y} f(x^*,y)$. Since $\{y\in Y: f(x^*,y) \geq l_{x^*}\}$ is bounded, we know from the continuity of $f(x^*,\cdot)$ that the maximum can be achieved. Let $y^*\in \argmax_{y\in Y} f(x^*,y)$. It is easy to check that $(x^*,y^*)$ is a global minimax point.
\end{proof}

Specifically, if both $X$ and $Y$ are bounded, then all conditions in (i) of Proposition \ref{Prop1} hold. Thus problem \eqref{minmax} has a global minimax point. However, a local minimax point may not exist even $X$ and $Y$ are bounded (see Example \ref{Eg2}). Also, a global minimax point may not be a local minimax point (see Example \ref{Eg2}). The following example tells that the global and local minimax points exist but (local) saddle points do not.

\begin{example}[{\cite[Figure 1]{JNJ2020local}}]\label{Eg1}
Let $n=m=1$ and $X = Y = [-1,1]$. Consider $f(x,y)=-x^2 + 5xy -y^2$. Note that
\begin{equation*}
\varphi(x)= \max_{y\in[-1,1]} (-x^2+5xy-y^2) =
\begin{cases}
-x^2 - 5x -1 , & x\in \left[-1, -\frac{2}{5}\right];\\
\frac{21}{4}x^2, & x\in \left[-\frac{2}{5}, \frac{2}{5}\right];\\
-x^2 + 5x - 1, & x\in \left[\frac{2}{5}, 1\right].\\
\end{cases}
\end{equation*}
It is not difficult to examine that $\min_{x\in [-1,1]} \varphi(x) = 0$ when $x=0$. In this case, $y=0$. Therefore, $(0,0)$ is a global minimax point. Moreover, let $\delta_0=\frac{2}{5}$ and $\tau(\delta)=\frac{5}{2}\delta$ in Definition \ref{Def4}. Then for any $\delta\leq \delta_0$, $(x,y)\in [-1,1]\times [-1,1]$ satisfying $\abs{x}\leq \delta$ and $\abs{y}\leq \delta$, we have
$$\max_{y'\in\{y\in Y: \abs{y}\leq \frac{5}{2}\delta\}}f(x,y') = \frac{21}{4}x^2$$
when $y= \frac{5}{2}x$. Thus, we obtain
$$-y^2=f(0,y) \leq f(0,0)=0 \leq \max_{y'\in\{y\in Y: \abs{y}\leq \frac{5}{2}\delta\}}f(x,y') = \frac{21}{4}x^2,$$
which implies that $(0,0)$ is also a local minimax point.

Note that the solutions of $\max_{y\in[-\delta,\delta]}\min_{x\in[-\delta,\delta]} f(x,y)$ are $(\delta,0)$ and $(-\delta,0)$ for any $\delta\in(0,1]$. Thus, we have
\begin{equation}
\label{gs5}
\max_{y\in[-\delta,\delta]}\min_{x\in[-\delta,\delta]} f(x,y) = -\delta^2 \neq 0 = \min_{x\in[-\delta,\delta]}\max_{y\in[-\delta,\delta]} f(x,y),
\end{equation}
which implies that $(0,0)$ is neither a saddle point (i.e., \eqref{gs5} holds with $\delta=1$) nor a local saddle point (i.e., \eqref{gs5} holds with sufficiently small $\delta$).

This example tells us (local) saddle points are inadequate to characterize the optimality of nonconvex-nonconcave min-max problems. Global and local minimax points defined in Definitions \ref{Def3} and \ref{Def4} respectively are good supplements of (local) saddle points.
\end{example}

\section{Optimality and stationarity}\label{Sec3}

In this section, we first discuss the first-order and second-order  optimality conditions when $f$ in \eqref{minmax} is nonsmooth. The smooth case is considered as a special case of the nonsmooth ones when the directional derivatives can be represented by Fr\'echet derivatives. The results extend  the study of necessary optimality conditions of min-max problems in \cite{JNJ2020local}, where the unconstrained smooth min-max problems are addressed. We illustrate these theoretical results by three examples.

To proceed further, we give the description of tangents to convex sets.

\begin{lemma}[{\cite[Theorem 6.9]{RW2009variational}}]\label{Lem3}
If $\Omega\subseteq \mathbb{R}^n$ is convex and $\bar{x}\in \Omega$, then
\begin{align*}
\mathcal{T}_{\Omega}(\bar{x})=\cl\{w: \exists \lambda >0~\text{with}~ \bar{x} + \lambda w \in \Omega\}, \,
\mathrm{int}\left(\mathcal{T}_{\Omega}(\bar{x}) \right)=\{w: \exists \lambda >0~\text{with}~ \bar{x} + \lambda w \in \mathrm{int}(\Omega)\}.
\end{align*}
\end{lemma}

Denote
$$\mathcal{T}_{\Omega}^\circ(\bar{x}):=\{w: \exists \lambda >0~\text{with}~ \bar{x} + \lambda w \in \Omega \}.$$
 It is not difficult to verify that $\mathcal{T}_{\Omega}(\bar{x})$, $\mathrm{int}(\mathcal{T}_{\Omega}(\bar{x}))$ and $\mathcal{T}_{\Omega}^\circ(\bar{x})$ are convex cones if $\Omega$ is convex. Moreover, we have the following relationship
\begin{equation*}
\mathrm{int}\left(\mathcal{T}_{\Omega}(\bar{x}) \right) \subseteq \mathcal{T}_{\Omega}^\circ(\bar{x}) \subseteq \mathcal{T}_{\Omega}(\bar{x}).
\end{equation*}
If $\Omega$ is polyhedral, then $\mathcal{T}_{\Omega}^\circ(\bar{x}) = \mathcal{T}_{\Omega}(\bar{x})$.

\subsection{Nonsmooth case}

In this subsection, we give some necessary optimality conditions of local minimax points of problem \eqref{minmax} for the case that $f$ is not differentiable, which lead to definitions of the first-order and second-order d-stationary points. For this purpose, we introduce some definitions for nonsmooth analysis.

Let $g:\mathbb{R}^n \to \mathbb{R}$. The (first-order) \emph{subderivative} $\d g(x)(v)$ at $x\in \mathbb{R}^n$ for $v\in\mathbb{R}^n$ is defined as \cite[Definition 8.1]{RW2009variational}
$$\d g(x)(v) :=\liminf_{v'\to v, t\downarrow 0} \frac{g(x+ t v') - g(x)}{t}.$$
The function $g$ is \emph{semidifferentiable} at $x$ for $v$ \cite[Definition 7.20]{RW2009variational} if the (possibly infinite) limit
$$\lim_{v'\to v, t\downarrow 0} \frac{g(x+ t v') - g(x)}{t}$$
exists. Further, if the above limit exists for every $v\in \mathbb{R}^n$, we say that $g$ is semidifferentiable at $x$. It is easy to see that if $g$ is Lipschitz continuous in a neighborhood of $x$, then this limit is finite.

There are two types of second-order subderivatives \cite[Definition 13.3]{RW2009variational}. The second-order subderivative at $x\in \mathbb{R}^n$ for $w$ and $v$ is
$$\d^2g(x|v)(w):=\liminf_{w'\to w, t\downarrow 0} \frac{g(x+ t w') - g(x) - t \inp{v,w'}}{\frac{1}{2}t^2}.$$
The second-order subderivative at $x\in \mathbb{R}^n$ for $w$ (without mention of $v$) is
$$\d^2g(x)(w):=\liminf_{w'\to w, t\downarrow 0} \frac{g(x+ t w') - g(x) - t \d g(x)(w')}{\frac{1}{2}t^2}.$$

We say that $g$ is \emph{twice semidifferentiable} at $x$ if it is semidifferentiable at $x$ and the (possibly infinite) limit
$$\lim_{w'\to w, t\downarrow 0} \frac{g(x+ t w') - g(x) - t \d g(x)(w')}{\frac{1}{2}t^2}$$
exists for any $w\in \mathbb{R}^n$.

The one-side \emph{directional derivative} $g'(x;v)$ at $x\in \mathbb{R}^n$ along the direction $v\in\mathbb{R}^n$ is defined as
$$g'(x;v):=\lim_{t\downarrow 0} \frac{g(x+ t v) - g(x)}{t}.$$
The function $g$ is directionally differentiable at $x$ if $g'(x;v)$ exists for all directions $v\in\mathbb{R}^n$. If $g$ is locally Lipschitz continuous near $x$, then semidifferentiability at $x$ is equivalent to directional differentiability at $x$.

The \emph{second-order directional derivative} of $g$ at $x\in \mathbb{R}^n$ along the direction $v\in\mathbb{R}^n$  is defined as \cite[Chapter 13.B]{RW2009variational}
$$g^{(2)}(x;v):=\lim_{t\downarrow 0} \frac{g(x+ t v) - g(x) - t g'(x;v)}{\frac{1}{2}t^2}.$$

Obviously, if $g$ is semidifferentiable at $x$, then $\d g(x)(v)=g'(x;v)$; if $g$ is twice semidifferentiable at $x$, then $\d^2 g(x)(w)=g^{(2)}(x;w)$.

As a generalization of classical directional derivatives, \emph{the (Clarke) generalized directional derivative} of $g$ at $x\in \mathbb{R}^n$ along the direction $v \in \mathbb{R}^n$ is defined as \cite[Section 2.1]{C1990optimization}
$$g^{\comp}(x;v) := \limsup_{x' \to x \atop  t\downarrow 0} \frac{g(x'+ tv) - g(x') }{ t}.$$
We say that $g$ is \emph{Clarke regular} at $x$ \cite[Definition 2.3.4]{C1990optimization} if $g'(x;v)$ exists and $g^{\comp}(x;v)=g'(x;v)$ for all $v$. By using the generalized directional derivative, we can define the (Clarke)\emph{ generalized subdifferential} as
$$\partial g(x):= \{z\in \mathbb{R}^n: \inp{z,v} \leq g^{\comp}(x;v)~ \forall v\in \mathbb{R}^n\}.$$

The \emph{generalized second-order directional derivative} of  $g$ at $x\in \mathbb{R}^n$ along the direction $(u, v) \in \mathbb{R}^n\times \mathbb{R}^n$ is defined as \cite[Definition 1.1]{CC1990generalized} and \cite[Theorem 13.52]{RW2009variational}
$$g^{\comp\comp}(x;u,v) := \limsup_{x' \to x \atop  t\downarrow 0, \delta\downarrow 0} \frac{g(x'+ \delta u + t v) - g(x'+ \delta u) - g(x'+t v) + g(x')}{\delta t}.$$

Especially, when $u=v$, we write $g^{\comp\comp}(x;v,v)$ as $g^{\comp\comp}(x;v)$ for simplicity.

\begin{remark}
When $f$ is continuously differentiable at $(\hat{x},\hat{y})$,  $f_x^{\comp}(\hat{x},\hat{y};v) = \d_x f(\hat{x},\hat{y}) (v)= \nabla_x f(\hat{x},\hat{y})^\top v$ and $f_y^{\comp}(\hat{x},\hat{y};w) =\d_y f(\hat{x},\hat{y}) (w) = \nabla_y f(\hat{x},\hat{y})^\top w$ (\cite[Exercise 8.20]{RW2009variational}). Moreover, if $f$ is twice continuously differentiable at $(\hat{x},\hat{y})$, we know from \cite[Example 13.8, Proposition 13.56]{RW2009variational} that $f_x^{\comp\comp}(\hat{x},\hat{y};v) = \d_x^2 f(\hat{x},\hat{y}) (v)= v^\top \nabla_x^2 f(\hat{x},\hat{y}) v$ and $f_y^{\comp\comp}(\hat{x},\hat{y};w) = \d_y^2 f(\hat{x},\hat{y}) (w) = w^\top \nabla_y^2 f(\hat{x},\hat{y}) w$.

\end{remark}

\begin{example}\label{Eg3}
Consider a two-layer neural network with the ReLU activation function as follows:
$$F(W,b):=\rho(W_2(W_1\xi+b_1)_+ + b_2),$$
where $\xi\in \mathbb{R}^s$, $W_1\in \mathbb{R}^{s_1\times s}$, $b_1\in \mathbb{R}^{s_1}$, $W_2\in \mathbb{R}^{s_2\times s_1}$, $b_2\in \mathbb{R}^{s_2}$, $\rho:\mathbb{R}^{s_2} \to \mathbb{R}$
is a continuously differentiable function, $W=(W_1,W_2)$ and $b=(b_1,b_2)$. Obviously, $F$ is locally Lipschitz continuous. Consider
\begin{align*}
&F'(W,b;\overline{W},\bar{b})=\lim_{t\downarrow 0}\frac{F(W + t \overline{W},b+ t \bar{b}) - F(W,b)}{t} \\
&=\lim_{t\downarrow 0}\frac{\rho((W_2 + t\overline{W}_2)((W_1 + t\overline{W}_1)\xi+b_1+ t\overline{b}_1)_+ + b_2 + t\overline{b}_2) - \rho(W_2(W_1\xi+b_1)_+ + b_2)}{t}
\end{align*}
and
\begin{small}
\begin{align*}
&~~~\lim_{t\downarrow 0}\frac{(W_2 + t\overline{W}_2)((W_1 + t\overline{W}_1)\xi+b_1+ t\overline{b}_1)_+ + b_2 + t\overline{b}_2 - (W_2(W_1\xi+b_1)_+ + b_2)}{t} \\
&= \lim_{t\downarrow 0}\frac{W_2 \left( ((W_1 + t\overline{W}_1)\xi+ b_1+ t\overline{b}_1)_+ - (W_1\xi+b_1)_+\right) + t\left( \overline{W}_2((W_1 + t\overline{W}_1)\xi+ b_1+ t\overline{b}_1)_+ + \overline{b}_2 \right) }{t} \\
&= W_2 \left( \lim_{t\downarrow 0}\frac{ ((W_1 + t\overline{W}_1)\xi+ b_1+ t\overline{b}_1)_+ - (W_1\xi+b_1)_+ }{t} \right) +  \overline{W}_2(W_1\xi+ b_1)_+  + \overline{b}_2.
\end{align*}
\end{small}
Since for $i=1,\cdots,s_1$,
\begin{align*}
&\lim_{t\downarrow 0}\frac{ ((W_1^i + t\overline{W}_1^i)^\top \xi+ b_1^i+ t\overline{b}_1^i)_+ - ((W_1^i)^\top\xi+b_1^i)_+ }{t} \\
&=
\begin{cases}
(\overline{W}_1^i)^\top \xi + \bar{b}_1^i & \text{if}~ (W_1^i)^\top \xi+ b_1^i>0 ~\text{or}~ (W_1^i)^\top \xi+ b_1^i=0~\text{and}~ (\overline{W}_1^i)^\top \xi + \bar{b}_1^i>0; \\
0 &   \text{if}~ (W_1^i)^\top \xi+ b_1^i< 0 ~\text{or}~ (W_1^i)^\top \xi+ b_1^i=0~\text{and}~ (\overline{W}_1^i)^\top \xi + \bar{b}_1^i\leq 0,
\end{cases}
\end{align*}
where $\overline{W}_1^i$ and $W_1^i$ are the $i$th row vectors of $\overline{W}_1$ and $W_1$ respectively and  $\bar{b}_1^i$ and $b_1^i$ are the $i$th components of $\bar{b}_1$ and $b_1$ respectively, we have that the following limit
$$\Upsilon:=W_2 \left( \lim_{t\downarrow 0}\frac{ ((W_1 + t\overline{W}_1)\xi+ b_1+ t\overline{b}_1)_+ - (W_1\xi+b_1)_+ }{t} \right) +  \overline{W}_2(W_1\xi+ b_1)_+  + \overline{b}_2$$
exists. Therefore, we have that $F$ is semidifferentiable based on the locally Lipschitz continuity.

If, moreover, $\rho$ is twice continuously differentiable, we have
\begin{align*}
\d^2F(W,b)(\overline{W},\bar{b})&=\liminf_{t\downarrow 0 \atop \overline{W}'\to \overline{W},\bar{b}' \to \bar{b}}\frac{F(W + t \overline{W}',b+ t \bar{b}') - F(W,b) - t\d F(W,b)(\overline{W}',\bar{b}')}{\frac{1}{2}t^2} \\
&= \Upsilon^\top  \nabla^2 \rho( W_2(W_1\xi+b_1)_+ + b_2 )  \Upsilon,
\end{align*}
which implies that $F$ is twice semidifferentiable.

\end{example}

The following lemma tells the necessary optimality conditions for an unconstrained minimization problem by using subderivatives.

\begin{lemma}[{\cite[Theorem 10.1 \& Theorem 13.24]{RW2009variational}}]\label{Lem5}
Let $g:\mathbb{R}^n \to (-\infty,+\infty]$ be a proper extended-valued function. If $\bar{x}$ is a local minimizer of $g$ over $\mathbb{R}^n$, then $\d g(\bar{x})(v)\geq 0$ and $\d^2g(\bar{x}|0)(v)\geq 0$ for any $v\in \mathbb{R}^n$.
\end{lemma}

The following lemma shows that we can replace $\d^2g(\bar{x}|0)(v)\geq 0$ by $\d^2g(\bar{x})(v)\geq 0$ under certain mild conditions.

\begin{lemma}\label{Lem6}

Let $g:\mathbb{R}^n \to (-\infty,+\infty]$ be twice semidifferentiable at $\bar{x}$. If $\d g(\bar{x})(v) = 0$, then $\d^2g(\bar{x}|0)(v)=\d^2g(\bar{x})(v)$.
\end{lemma}

\begin{proof}
Let $\d g(\bar{x})(v) = 0$. Note that
\begin{align*}
\d^2g(\bar{x})(v)&=\liminf_{v'\to v, t\downarrow 0} \frac{g(\bar{x}+ t v') - g(\bar{x}) - t \d g(\bar{x})(v')}{\frac{1}{2}t^2}\\
&=\lim_{v'\to v, t\downarrow 0} \frac{g(\bar{x}+ t v') - g(\bar{x}) - t \d g(\bar{x})(v')}{\frac{1}{2}t^2}\\
&=\lim_{t\downarrow 0} \frac{g(\bar{x}+ t v) - g(\bar{x}) - t \d g(\bar{x})(v)}{\frac{1}{2}t^2}\\
&=\lim_{t\downarrow 0} \frac{g(\bar{x}+ t v) - g(\bar{x})}{\frac{1}{2}t^2}=\d^2g(x|0)(v),
\end{align*}
where the second equality follows from the twice semidifferentiability of $g$ at $\bar{x}$ and the third equality follows from the existence of the limit.
\end{proof}

\begin{lemma}[{\cite[Theorem 8.2]{RW2009variational}}]\label{Lem7}
For the indicator function $\delta_\mathcal{X}$ of a set $\mathcal{X}\subseteq \mathbb{R}^n$ and any point $x\in \mathcal{X}$, one has
$\d \delta_\mathcal{X} (x)(v) = \delta_{\mathcal{T}_\mathcal{X}(x)}(v)$ for any $v\in \mathbb{R}^n$.
\end{lemma}

A function $g:\mathbb{R}^n\to \mathbb{R}$ is called \emph{positively homogeneous of degree $p > 0$} if $g(\lambda w) = \lambda^p g(w)$ for all $\lambda > 0$ and $w\in \mathbb{R}^n$ (see \cite[Definition 13.4]{RW2009variational}).

The following lemma shows the expansion of a function via subderivatives.

\begin{lemma}[{\cite[Theorem 7.21 \& Exercise 13.7]{RW2009variational}}]\label{Lem9}
Let $g:\mathbb{R}^n\to \mathbb{R}$. Then
\begin{enumerate}
\item[(i)] $g$ is semidifferentiable at $\bar{x}$ if and only if
$$g(x)=g(\bar{x}) + \d g(\bar{x})(x-\bar{x}) + o(\norm{x-\bar{x}}),$$
where $\d g(\bar{x})(\cdot)$ is a finite, continuous, positively homogeneous function.

\item[(ii)] Suppose that $g$ is semidifferentiable at $\bar{x}$. Then $g$ is twice semidifferentiable at $\bar{x}$ if and only if
$$g(x)=g(\bar{x}) + \d g(\bar{x})(x-\bar{x}) + \frac{1}{2}\d^2 g(\bar{x})(x-\bar{x}) +  o(\norm{x-\bar{x}}^2),$$
where $\d^2 g(\bar{x})(\cdot)$ is a finite, continuous, positively homogeneous of degree $2$ function.
\end{enumerate}
\end{lemma}

 The following lemma gives the first-order and second-order optimality conditions for minimizing a semidifferentiable function, which extends \cite[Proposition 2.3]{CCHP2020study} with a polyhedral set $\mathcal{X}$. For completeness, we give a simple proof.
\begin{lemma}\label{Prop7}
Let $\mathcal{X}\subseteq \mathbb{R}^n$ be a closed and convex set, $g: \mathbb{R}^n\to \mathbb{R}$ be semidifferentiable at $\bar{x}\in \mathcal{X}$, and $\bar{x}$ be a local minimizer of $g$ over $\mathcal{X}$. Then
 $\d g(\bar{x})(v)\geq 0$ for all $v\in \mathcal{T}_{\mathcal{X}}(\bar{x})$.
Moreover, if $g$ is twice semidifferentiable at $\bar{x}$, then $\d^2g(\bar{x})(v)\geq 0$ for all $v\in  \mathcal{T}_{\mathcal{X}}^\circ(\bar{x}) \cap \{v: \d g(\bar{x})(v) =0\}$.
\end{lemma}

\begin{proof}
From Lemma \ref{Lem5} and Lemma \ref{Lem7}, we have
$$0\leq \d (g(\bar{x}) + \delta_{\mathcal{X}}(\bar{x}))(v)= \d g(\bar{x})(v) + \delta_{\mathcal{T}_{\mathcal{X}}(\bar{x})}(v) = \d g(\bar{x})(v).$$
Moreover, if $g$ is twice semidifferentiable at $\bar{x}$, for  $v\in  \mathcal{T}_{\mathcal{X}}^\circ(\bar{x}) \cap \{v: \d g(\bar{x})(v) =0\}$,
from Lemma \ref{Lem6}, we have
\begin{align*}
0&\leq
\liminf_{v'\to v, t\downarrow 0} \frac{g(\bar{x}+ t v') + \delta_{\mathcal{X}}(\bar{x}+ t v') - g(\bar{x}) - \delta_{\mathcal{X}}(\bar{x}) - t \d g(\bar{x})(v)}{\frac{1}{2} t^2} \\
&\leq \liminf_{t\downarrow 0} \frac{g(\bar{x}+ t v) + \delta_{\mathcal{X}}(\bar{x}+ t v) - g(\bar{x}) - \delta_{\mathcal{X}}(\bar{x}) - t \d g(\bar{x})(v)}{\frac{1}{2} t^2} \\
&= \lim_{t\downarrow 0} \frac{g(\bar{x}+ t v) - g(\bar{x}) - t \d g(\bar{x})(v) }{\frac{1}{2} t^2}\\
&= \d^2 g(\bar{x})(v).
\end{align*}
\end{proof}

The following lemma gives a description of the generalized second-order directional derivative by using directional
derivatives.

\begin{lemma}[{\cite[Proposition 1.3]{CC1990generalized}}]\label{Lem10}
Let $g:\mathbb{R}^n\to \mathbb{R}$ be a continuous function that admits a directional
derivative at every point near $x$. Then $g^{\comp\comp}(x;u,v)$ is the generalized directional derivative of $g'(\cdot,v)$ at $x$ along direction $u$, that is
$$g^{\comp\comp}(x;u,v) = \limsup_{x' \to x \atop  t\downarrow 0} \frac{g'(x'+ t u;v) - g'(x';v) }{ t}.$$
\end{lemma}

\begin{remark}\label{Rem1}
Note that
%
\begin{align*}
g^{\comp\comp}(x;v) \geq  \lim_{ t\downarrow 0} \frac{g(x+t v + t v) - g(x+ t v) - g(x+t v) + g(x)}{t^2} = g^{(2)}(x;v).
\end{align*}

Recall that $g:\mathbb{R}^n\to \mathbb{R}$ is \emph{twice subregular} at $x$ \cite[Definition 3.1]{CC1990generalized} if the limit
$$\lim_{ t\downarrow 0, \delta\downarrow 0} \frac{g(x+ \delta u + t v) - g(x+ \delta u) - g(x+t v) + g(x)}{\delta t}$$
exists and the above limit equals to $g^{\comp\comp}(x;u,v)$. Thus, we know that $g^{\comp\comp}(x;v) = g^{(2)}(x;v)$ if $g$ is twice subregular at $x$.
\end{remark}

Now we are ready to give the main results of this subsection.

\begin{theorem}\label{Th5}
Let the tuple $(\hat{x},\hat{y})\in X\times Y$ be a local minimax point of problem \eqref{minmax}.

\begin{enumerate}
\item[(i)] If $f$ is semidifferentiable at $(\hat{x},\hat{y})$, then
\begin{subequations}
\label{NonS1st}
\begin{align}
&f_x^{\comp}(\hat{x},\hat{y};v) \geq 0~\text{for all}~v\in \mathcal{T}_{X}(\hat{x}),\label{NonS1st-1}\\
&\d_y f(\hat{x},\hat{y}) (w) \leq 0~\text{for all}~w\in \mathcal{T}_{Y}(\hat{y}),\label{NonS1st-2}
\end{align}
\end{subequations}
where $f_x^{\comp}(\hat{x}, \hat{y};v)$ denotes the generalized directional derivative of $f$ with respect to $x$ at $\hat{x}$ along the direction $v$ for fixed $\hat{y}$.

\item[(ii)] Assume, further, that $f$ is twice semidifferentiable at $(\hat{x},\hat{y})$ and $f$ is Clarke regular in a neighborhood of $(\hat{x},\hat{y})$. Then
    \end{enumerate}
\begin{subequations}
\label{NonS2ed}
\begin{align}
&f_x^{\comp\comp}(\hat{x},\hat{y};v) \geq 0~\text{for all}~v\in \mathcal{T}_{X}^\circ(\hat{x}) \cap \{v: \exists \delta>0, \d_x f(\hat{x},y') (v) =0,\forall y' \in \mathbb{B}(\hat{y},\delta)\cap Y\}, \label{NonS2ed-1}\\
&\d_y^2 f(\hat{x},\hat{y}) (w) \leq 0~\text{for all}~w\in \mathcal{T}_{Y}^\circ(\hat{y}) \cap \{w: \d_y f(\hat{x},\hat{y}) (w) =0 \},\label{NonS2ed-2}
\end{align}
\end{subequations}
where $f_x^{\comp\comp} (\hat{x},\hat{y}; v)$ denotes the generalized second-order directional derivative of $f$ with respect to $x$ at $\hat{x}$ along the direction $(v,v)$ for fixed $\hat{y}$.

\end{theorem}

\begin{proof}
\eqref{NonS1st-2} and \eqref{NonS2ed-2} directly follow from Proposition \ref{Prop7}. Therefore, we only focus on  \eqref{NonS1st-1} and \eqref{NonS2ed-1}, respectively.

(i) Since $(\hat{x},\hat{y})$ is a local minimax point, there exist a $\delta_0>0$ and a function $\tau:\mathbb{R}_+\rightarrow \mathbb{R}_+$ satisfying $\tau(\delta)\rightarrow 0$ as $\delta\rightarrow 0$, such that for any $\delta\in (0,\delta_0]$ and $(x,y)\in X\times Y$ satisfying $\norm{x-\hat{x}}\leq \delta$ and $\norm{y-\hat{y}}\leq \delta$, we have
\begin{equation}\label{stationary1}
f(\hat{x},y) \leq f(\hat{x},\hat{y}) \leq \max_{y'\in\{y\in Y: \norm{y-\hat{y}}\leq \tau(\delta)\}}f(x,y').
\end{equation}

For any $v\in\mathcal{T}_X(\hat{x})$, according to the convexity of $X$, there exist $\{v^k\}_{k\geq 1}$ with $v^k\to v$ as $k\to \infty$ and $\{t_k\}_{k\geq 1}$ with $t_k \downarrow 0$ as $k\to \infty$, such that $x^k:=\hat{x}+ t_kv^k \in X$ (see Lemma \ref{Lem3}). Let $\delta_k=\norm{x^k -\hat{x}}$ and $\tilde{y}^k$ be defined by
\begin{equation}
\label{gs10}
\tilde{y}^k\in \argmax_{y' \in \{y\in Y:  \norm{y-\hat{y}}\leq \tau(\delta_k)\}} f(x^k,y').
\end{equation}
Obviously, $\delta_k\to 0$ and $\norm{\tilde{y}^k \to \hat{y}}\to 0$ as $k\to \infty$. According to the second inequality of (\ref{stationary1}), we have (for sufficiently large $k$) that
\begin{equation}
\label{gs18}
\begin{aligned}
0 \leq f(x^k, \tilde{y}^k) - f(\hat{x}, \hat{y}) &= f(x^k, \tilde{y}^k) - f(\hat{x}, \tilde{y}^k) +  f(\hat{x}, \tilde{y}^k) - f(\hat{x}, \hat{y}) \\
&\leq f(x^k, \tilde{y}^k) - f(\hat{x}, \tilde{y}^k) \\
& \in \inp{\partial f(\tilde{x}^k, \tilde{y}^k), \begin{pmatrix}t_k v^k\\ 0 \end{pmatrix}},
\end{aligned}
\end{equation}
where the last inclusion follows from the mean-value theorem \cite[Theorem 2.3.7]{C1990optimization} and $\tilde{x}^k$ lies in the segment between $x^k$ and $\hat{x}$. Thus, by dividing $t_k$ in both sides and letting $k\to \infty$, due to the upper semicontinuity of $\partial f(\cdot,\cdot)$ (see \cite[Proposition 2.1.5]{C1990optimization}), we obtain
\begin{align*}
0 &\leq \sup_{\zeta\in \partial f(\hat{x}, \hat{y})} \inp{\zeta,\begin{pmatrix}v\\ 0 \end{pmatrix}}
= f^{\comp} (\hat{x}, \hat{y};v,0)
= f_x^{\comp}(\hat{x}, \hat{y};v),
\end{align*}
where $f_x^{\comp}(\hat{x}, \hat{y};v)$ denotes the Clarke generalized directional derivative of $f$ with respect to $x$ at $\hat{x}$  along the direction $v$ for fixed $\hat{y}$.


(ii) Let $v \in \mathcal{T}_{X}^\circ(\hat{x}) \cap \{v: \exists \delta>0, \d_x f(\hat{x},y') (v) =0,\forall y' \in \mathbb{B}(\hat{y},\delta)\cap Y\}$. Then there exists a sequence $\{t_k\}_{k\geq 1}$ with $t_k \downarrow 0$, such that $x^k:=\hat{x} + t_kv \in X$. Let $\delta_k=\norm{x^k-\hat{x}}$, and $\tilde{y}^k$ be defined in \eqref{gs10}.

From mean-value theorem, there is $\zeta_k\in(0,t_k)$ such that
\begin{equation}
\label{gs12}
f(\hat{x}+t_k v,\tilde{y}^k) - f(\hat{x}, \tilde{y}^k) \in   \partial f(\hat{x}+ \zeta_k v,\tilde{y}^k) \begin{pmatrix}v\\0 \end{pmatrix}.
\end{equation}
Similar to \eqref{gs18}, we have $f(\hat{x} + t_k v, \tilde{y}^k) - f(\hat{x}, \tilde{y}^k) \geq 0$. Then, according to the Clarke regularity of $f$ near $(\hat{x},\hat{y})$ and the definition of generalized directional derivative, we have from \eqref{gs12} that
\begin{align*}
0 &\leq \limsup_{k\to \infty} \frac{ f'(\hat{x}+ \zeta_k v,\tilde{y}^k;v,0) }{\zeta_k} = \limsup_{k\to \infty} \frac{ f'(\hat{x}+ \zeta_k v,\tilde{y}^k;v,0) - f'(\hat{x}, \tilde{y}^k;v,0) }{\zeta_k}\\
& \leq f^{\comp\comp} (\hat{x},\hat{y}; v,0) = f_x^{\comp\comp} (\hat{x},\hat{y}; v),
\end{align*}
where the first equality follows from $f'(\hat{x}, \tilde{y}^k;v,0)=0$ for sufficiently large $k$, the second inequality follows from Lemma \ref{Lem10} and $f_x^{\comp\comp} (\hat{x},\hat{y}; v)$ denotes the generalized second-order directional derivative of $f$ with respect to $x$  at $\hat{x}$ along the direction $(v,v)$ for fixed $\hat{y}$.
\end{proof}

\begin{remark}\label{Rem2}
We know from \cite[Page 10]{C1990optimization}  that, for any $v$,
$f_x^{\comp}(\hat{x},\hat{y};v)=\max_{z\in \partial_x f(\hat{x},\hat{y})} \inp{z,v}$.
Thus, \eqref{NonS1st-1} can be equivalently reformulated as
$\max_{z\in \partial_x f(\hat{x},\hat{y})} \inp{z,v} \geq 0,~\forall v\in \mathcal{T}_{X}(\hat{x})$,
which, based on the definition of normal cone, is equivalent to $0\in \partial_x f(\hat{x},\hat{y}) + \mathcal{N}_X(\hat{x})$.

Generally, \eqref{NonS1st-2} implies the Clarke stationary condition
$
0\in -\partial_y f(\hat{x},\hat{y}) + \mathcal{N}_Y(\hat{y}),
$
but not vice versa. Moreover, by using the (generalized) directional derivatives, we can establish the second-order necessary optimality conditions for the nonsmooth case. Therefore, the (generalized) directional derivatives are employed in Theorem \ref{Th5}.
\end{remark}

\begin{remark}
It is noteworthy that the necessary optimality conditions \eqref{NonS1st-1}-\eqref{NonS1st-2} and \eqref{NonS2ed-1}-\eqref{NonS2ed-2} with respect to $x$ and $y$ are not symmetric. Generally, \eqref{NonS1st-1} and \eqref{NonS2ed-1} are weaker than
\begin{equation}
\label{gs16}
\d_x f(\hat{x},\hat{y};v) \geq 0~\text{for all}~v\in \mathcal{T}_{X}(\hat{x})
\end{equation}
and
\begin{equation}
\label{gs17}
\d_x^2 f(\hat{x},\hat{y};v) \geq 0~\text{for all}~v\in \mathcal{T}_{X}^\circ(\hat{x}) \cap \{v: \d_x f(\hat{x},\hat{y}) (v) =0 \},
\end{equation}
respectively, because $f_x^{\comp}(\hat{x},\hat{y};v) \geq \d_x f(\hat{x},\hat{y};v)$, $f_x^{\comp\comp}(\hat{x},\hat{y};v)\geq \d_x^2 f(\hat{x},\hat{y};v)$ (Remark \ref{Rem1}) and
$$\mathcal{T}_{X}^\circ(\hat{x}) \cap \{v: \exists \delta>0, \d_x f(\hat{x},y') (v) =0,\forall y' \in \mathbb{B}(\hat{y},\delta)\cap Y\} \subseteq \mathcal{T}_{X}^\circ(\hat{x}) \cap \{v: \d_x f(\hat{x},\hat{y}) (v) =0 \}.$$
The main reason is that a local minimax point may not be a local saddle point. If we replace \eqref{NonS1st-1} and \eqref{NonS2ed-1} by \eqref{gs16} and \eqref{gs17} respectively, the necessary optimality conditions for local saddle points are derived.

If, in addition, $f$ is Clarke regular at $(\hat{x},\hat{y})$, \eqref{NonS1st-1} can be replaced by (\ref{gs16}).

If, in addition, $f$ is twice subregular at $(\hat{x},\hat{y})$, \eqref{NonS2ed-1} can be replaced by
$$\d_x^2 f(\hat{x},\hat{y}) (v) \geq 0~\text{for all}~v\in \mathcal{T}_{X}^\circ(\hat{x}) \cap \{v: \exists \delta>0, \d_x f(\hat{x},y') (v) =0,\forall y' \in \mathbb{B}(\hat{y},\delta)\cap Y\}.$$

\end{remark}

Based on Theorem \ref{Th5}, we define the first-order and second-order d-stationary points of min-max problems.

\begin{definition}\label{Def6}
We call that $(\hat{x},\hat{y})\in X\times Y$ is a first-order d-stationary point of problem \eqref{minmax} if it satisfies \eqref{NonS1st-1}-\eqref{NonS1st-2}. If $(\hat{x},\hat{y})$ also satisfies \eqref{NonS2ed-1}-\eqref{NonS2ed-2}, we call it a second-order d-stationary point of problem \eqref{minmax}.
\end{definition}

\begin{example}
Let $X=[-1,1]$ and $Y=[-1,1]$. We consider
$$\min_{x\in [-1,1]}\max_{y\in [-1,1]} f(x,y):=-\abs{x}^9 + \frac{3}{5}\abs{x}^3\abs{y}^3- \abs{y}^5.$$
Take $\tau(\delta)=\frac{3}{5}(\sqrt{\delta})^3$. Then for any $\abs{x}\leq \delta$ and $\abs{y}\leq \delta$ with sufficiently small $\delta\in(0,1)$ we have
\begin{align*}
-\abs{y}^5 = f(0,y) \leq f(0,0) &\leq  \max_{y\in [-\tau(\delta),\tau(\delta)]} -\abs{x}^9 + \frac{3}{5}\abs{x}^3\abs{y}^3- \abs{y}^5=-\abs{x}^9 + \frac{2}{5}\left( \frac{3}{5} \right)^4(\sqrt{\abs{x}})^{15}
\end{align*}
where $\pm\frac{3}{5}(\sqrt{\abs{x}})^3$ is the maximizer of the above maximization problem. This implies that $(0,0)$ is a local minimax point.
Obviously, $f(x,y)$ is not differentiable at $(0,0)$. In what follows, we examine the necessary optimality conditions in Theorem \ref{Th5}. Since $\mathcal{T}_{X}(0) = \mathbb{R}$ and $\mathcal{T}_{Y}(0) = \mathbb{R}$, we have for any $v\in \mathcal{T}_{X}(0)$ that
\begin{align*}
f_x^{\comp}(0,0;v)&= \limsup_{x'\to 0, t\downarrow 0}\frac{-\abs{x'+tv}^9 + \abs{x'}^9}{t}= 0,
\end{align*}
which implies that $f_x^{\comp}(0,0;v)= f'_x(0,0;v)$, i.e., the Clarke regularity holds.

Similarly, we have for any $w\in \mathcal{T}_{Y}(0)$ that
\begin{align*}
\d_y f(0,0) (w) &=  \liminf_{w'\to w, t\downarrow 0} \frac{f(0,t w') - f(0,0)}{t} =\liminf_{w'\to w, t\downarrow 0} \frac{-\abs{tw'}^5}{t} = 0.
\end{align*}

Next consider the second-order optimality conditions. Note that $\mathcal{T}_{X}^\circ(0) =\mathbb{R}$ and for any fixed $y'$, we have
\begin{align*}
\d_x f(0,y') (v) &= \liminf_{v'\to v, t\downarrow 0} \frac{f(t v', y') - f(0,y')}{t} \\
&= \liminf_{v'\to v, t\downarrow 0} \frac{-t^9\abs{ v'}^9 + \frac{3}{5}t^3\abs{v'}^3\abs{y'}^3 - \abs{y'}^5 + \abs{y'}^5 }{t} \\
&=0
\end{align*}
for any $v$, which implies that $\{v: \d_x f(0,y') (v) =0 \} = \mathbb{R}$. Thus, for any $\delta>0$
$$\mathcal{T}_{X}^\circ(0) \cap \{v: \d_x f(0,y') (v) =0,\forall y' \in \mathbb{B}(0,\delta)\cap Y\} = \mathbb{R}.$$
Notice that
\begin{align*}
f_x^{\comp\comp}(0,0;v) &= \limsup_{x' \to 0 \atop  t\downarrow 0, \delta\downarrow 0} \frac{f(x'+ \delta v + t v,0) - f(x'+ \delta v,0) - f(x'+t v,0) + f(x',0)}{\delta t} \\
&= \limsup_{x' \to 0 \atop  t\downarrow 0, \delta\downarrow 0} \frac{-\abs{x'+ \delta v + t v}^9 + \abs{x'+ \delta v}^9 + \abs{x'+t v}^9 - \abs{x'}^9 }{\delta t} \geq 0
\end{align*}
for any $v\in \mathbb{R}$. Similarly, we have $\mathcal{T}_{Y}^\circ(0) \cap \{w: \d_y f(0,0) (w) =0 \} =\mathbb{R}$ and
\begin{align*}
\d_y^2 f(0,0) (w) &= \liminf_{w'\to w, t\downarrow 0} \frac{f(0, t w') - f(0,0) - t \d_y f(0,0)(w')}{\frac{1}{2}t^2}
= \liminf_{w'\to w, t\downarrow 0} \frac{-\abs{t w'}^5}{\frac{1}{2}t^2} = 0
\end{align*}
for any $w\in \mathbb{R}$.
\end{example}

\subsection{Smooth case}
In this subsection, we consider the necessary optimality conditions of problem \eqref{minmax} when $f$ is (twice) continuously differentiable.
For any $(x,y)\in X\times Y$, denote
\begin{align*}
\Gamma_1^\circ(x,y)&=\{v\in \mathcal{T}_X^\circ(x): v\bot \nabla_x f(x,y)\}, \quad \,
\Gamma_1(x,y)=\{v\in \mathcal{T}_X(x): v\bot \nabla_x f(x,y)\},\\
\Gamma_2^\circ(x,y)&=\{w\in \mathcal{T}_Y^\circ(y): w\bot \nabla_y f(x,y)\}, \quad
\Gamma_2(x,y)=\{w\in \mathcal{T}_Y(y): w\bot \nabla_y f(x,y)\}.
\end{align*}

It is noteworthy that $\cl( \Gamma_1^\circ(x,y) ) \neq \Gamma_1(x,y)$ and $\cl( \Gamma_2^\circ(x,y) ) \neq \Gamma_2(x,y)$ generally even if we have $\cl( \mathcal{T}_X^\circ(x) ) = \mathcal{T}_X(x)$ and $\cl( \mathcal{T}_Y^\circ(y) ) = \mathcal{T}_Y(y)$. We summarize their relationships as follows.

\begin{lemma}
\label{Lem1}
Let $(x,y)\in X\times Y$. Then $\Gamma_1^\circ(x,y)$, $\Gamma_1(x,y)$, $\Gamma_2^\circ(x,y)$ and $\Gamma_2(x,y)$ are convex cones, and we have
$$\cl\Gamma_1^\circ(x,y) \subseteq \Gamma_1(x,y)~\text{and}~\cl\Gamma_2^\circ(x,y) \subseteq \Gamma_2(x,y).$$

Moreover, if $X$ and $Y$ are polyhedral, then $\Gamma_1^\circ(x,y)= \cl\Gamma_1^\circ(x,y) = \Gamma_1(x,y)$ and $\Gamma_2^\circ(x,y) = \cl\Gamma_2^\circ(x,y) = \Gamma_2(x,y)$.
\end{lemma}

The proof is based on Lemma \ref{Lem3}, and thus we omit it here.

\begin{theorem}\label{Th1}
Let $f$ be continuously differentiable and the tuple $(\hat{x},\hat{y})\in X\times Y$ be a local minimax point of problem \eqref{minmax}.
\begin{enumerate}
\item[(i)] Then it holds that
\begin{subequations}
\label{gs2}
\begin{align}
0&\in \nabla_x f(\hat{x},\hat{y}) + \mathcal{N}_{X}(\hat{x}),\label{gs2-1}\\
0&\in - \nabla_y f(\hat{x},\hat{y}) + \mathcal{N}_{Y}(\hat{y}).\label{gs2-2}
\end{align}
\end{subequations}

\item[(ii)]  Assume, further, that $f$ is twice continuously differentiable. Then
\begin{subequations}
\label{gs6}
\begin{align}
&\inp{v, \nabla_{xx}^2 f(\hat{x},\hat{y}) v} \geq 0~\text{for all}~v\in \cl \left\{\bar{v}: \exists \delta>0, \bar{v} \in \Gamma_1^\circ(\hat{x},y'), \forall y'\in \mathbb{B}(\hat{y},\delta) \right\},\label{gs6-1}\\
&\inp{w, \nabla_{yy}^2 f(\hat{x},\hat{y}) w} \leq 0~\text{for all}~w\in \cl\Gamma_2^\circ(\hat{x},\hat{y}).\label{gs6-2}
\end{align}
\end{subequations}
\end{enumerate}
\end{theorem}

\begin{proof}
(i) The proof is similar to Theorem \ref{Th5}. Here we give a simple proof of \eqref{gs2-1} and \eqref{gs6-1} for completeness. For any $x^k\overset{X}{\to}\hat{x}$ as $k\to \infty$, denote $\delta_k=\norm{x^k-\hat{x}}$ and $\tilde{y}^k$ is defined in \eqref{gs10}.
Obviously, $\delta_k\to 0$ and $\norm{\tilde{y}^k \to \hat{y}}\to 0$ as $k\to \infty$. From the continuous differentiability of $f$, we have
\begin{align*}
0 \leq f(x^k, \tilde{y}^k) - f(\hat{x}, \tilde{y}^k)
= \nabla f(\bar{x}^k, \tilde{y}^k)^\top \begin{pmatrix} x^k-\hat{x}\\ \tilde{y}^k-\tilde{y}^k \end{pmatrix}
= \nabla_x f(\hat{x}, \hat{y})^\top (x^k-\hat{x})   + o\left(\norm{x^k-\hat{x}}\right),
\end{align*}
where $\bar{x}^k$ is some point lying in the segment between $\hat{x}$ and $x^k$. 
Thus, we obtain
$$-\nabla_x f(\hat{x}, \hat{y})^\top (x^k-\hat{x}) \leq  o\left(\norm{x^k-\hat{x}}\right).$$
We know from \cite[Definition 6.3]{RW2009variational} that $-\nabla_x f(\hat{x}, \hat{y}) \in \mathcal{N}_X(\hat{x})$,
which verifies \eqref{gs2-1}.

(ii) We only need to prove that \eqref{gs6-1} holds with $v\in \Gamma_1^\circ(\hat{x},y')$ for all $y'\in \mathbb{B}(\hat{y},\delta)$ and some $\delta>0$. According to the definition of $\mathcal{T}_X^\circ(\hat{x})$, there exists a sequence $\{t_k\}_{k\geq 1}$ with $t_k \downarrow 0$ as $k\to \infty$, such that $x^k:=\hat{x} + t_kv \in X$. Let $\delta_k=t_k\norm{v}$, and $\tilde{y}^k$ is denoted in \eqref{gs10}. Similarly, we have that
\begin{align*}
0 & \leq  f(x^k, \tilde{y}^k) - f(\hat{x}, \tilde{y}^k) \\
&\overset{(a)}{=}\nabla_x f(\hat{x}, \tilde{y}^k)^\top (x^k-\hat{x}) + \frac{1}{2} (x^k-\hat{x})^\top \nabla_{xx}^2 f(\tilde{x}^k, \tilde{y}^k)(x^k-\hat{x}) \\
&\overset{(b)}{=}\nabla_x f(\hat{x}, \tilde{y}^k)^\top (x^k-\hat{x}) + \frac{1}{2} (x^k-\hat{x})^\top \nabla_{xx}^2 f(\hat{x}, \hat{y})(x^k-\hat{x}) + o\left(\norm{x^k-\hat{x}}^2\right),
\end{align*}
where (a) follows from Taylor's theorem for multivariate functions with Lagrange's remainder, and $\tilde{x}^k$ is some point lying in the segment between $\hat{x}$ and $x^k$; (b) follows from the twice continuous differentiability of $f$ and $\tilde{x}^k \to \hat{x}$ as $k\to \infty$. Thus, we obtain
\begin{align*}
t_k \nabla_x f(\hat{x}, \tilde{y}^k)^\top v + t_k^2 \frac{1}{2} v^\top \nabla_{xx}^2 f(\hat{x}, \hat{y})v + \norm{v}^2o(t_k^2)\geq 0.
\end{align*}
Since  $\nabla_x f(\hat{x}, \tilde{y}^k)^\top v = 0$ for sufficiently large $k$,
dividing by $t_k^2$ in both sides and letting $k\to \infty$, we complete the proof.
\end{proof}

\begin{remark}
Conversely, if these conditions in (ii) of Theorem \ref{Th1} hold except that $\cl \left\{w: \exists \delta>0, w \in \Gamma_1^\circ(\hat{x},y'), \forall y'\in \mathbb{B}(\hat{y},\delta) \right\}$ and $\cl\Gamma_2^\circ(\hat{x},\hat{y})$ are replaced by $\Gamma_1(\hat{x},\hat{y})$ and $\Gamma_2(\hat{x},\hat{y})$, respectively, and the inequality is strict when $v\neq 0$ and $w\neq 0$, then $(\hat{x},\hat{y})$ is a local saddle point. In that case, \eqref{gs6} together with \eqref{gs2} are the so-called second-order sufficient condition for a local saddle point.

If both $X$ and $Y$ are polyhedral, We can replace
$$\cl \left\{w: \exists \delta>0, w \in \Gamma_1^\circ(\hat{x},y'), \forall y'\in \mathbb{B}(\hat{y},\delta) \right\} \quad {\rm and} \quad  \cl\Gamma_2^\circ(\hat{x},\hat{y})$$
in Theorem \ref{Th1} by $\cl \left\{w: \exists \delta>0, w \in \Gamma_1(\hat{x},y'), \forall y'\in \mathbb{B}(\hat{y},\delta) \right\}$ and $\Gamma_2(\hat{x},\hat{y})$, respectively (see Lemma \ref{Lem1}).
\end{remark}

\begin{corollary}
Let $f$ be twice continuously differentiable. If, further, for local minimax point $(\hat{x},\hat{y})$, there exists an $\tau$ such that $\tau(\delta)=O(\delta)$ as $\delta\downarrow 0$, then \eqref{gs6-1} can be replaced by
\begin{equation*}
\inp{v, \nabla_{xx}^2 f(\hat{x},\hat{y}) v} \geq 0~\text{for all}~v\in \cl \Gamma_1^\circ(\hat{x},\hat{y}).
\end{equation*}
\end{corollary}

\begin{proof}
Let $0\neq v\in\Gamma_1^\circ(\hat{x},\hat{y})$. According to the definition of $\mathcal{T}_X^\circ(\hat{x})$, there exists a sequence $\{t_k\}_{k\geq 1}$ with $t_k \downarrow 0$ as $k\to \infty$, such that $x^k:=\hat{x} + t_kv \in X$. Let $\delta_k:=\norm{x^k -\hat{x}}$, and $\tilde{y}^k$ be denoted in \eqref{gs10}. Since $\tau(\delta)=O(\delta)$ as $\delta\downarrow 0$, we have $\norm{\tilde{y}^k - \hat{y}} =O\left(\norm{x^k - \hat{x}}\right)$ for sufficiently large $k$.
We know from the twice continuous differentiability of $f$ that
\begin{align*}
f(x^k, \tilde{y}^k) &=  f(\hat{x},\hat{y}) + \nabla_x f(\hat{x},\hat{y})^\top (x^k-\hat{x}) + \nabla_y f(\hat{x},\hat{y})^\top (\tilde{y}^k-\hat{y}) \\
&~~~+\frac{1}{2} (x^k-\hat{x})^\top \nabla_{xx}^2 f(\hat{x}, \hat{y})(x^k-\hat{x}) + (x^k-\hat{x})^\top \nabla_{xy}^2 f(\hat{x}, \hat{y})(\tilde{y}^k-\hat{y}) \\
&~~~+ \frac{1}{2} (\tilde{y}^k-\hat{y})^\top \nabla_{yy}^2 f(\hat{x}, \hat{y})(\tilde{y}^k-\hat{y}) + o\left(\norm{x^k-\hat{x}}^2 + \norm{\tilde{y}^k-\hat{y}}^2 \right), \\
f(\hat{x}, \tilde{y}^k) &=f(\hat{x},\hat{y}) + \nabla_y f(\hat{x},\hat{y})^\top (\tilde{y}^k-\hat{y}) + \frac{1}{2} (\tilde{y}^k-\hat{y})^\top \nabla_{yy}^2 f(\hat{x}, \hat{y})(\tilde{y}^k-\hat{y}) \\
&~~~+ o\left(\norm{\tilde{y}^k-\hat{y}}^2 \right).
\end{align*}
Using $t_k\nabla_x f(\hat{x}, \hat{y})^\top v =\nabla_x f(\hat{x}, \hat{y})^\top (x^k-\hat{x})= 0$  for  $ v\in\Gamma_1^\circ(\hat{x},\hat{y})$,  we have
\begin{align*}
0 &\leq f(x^k, \tilde{y}^k) - f(\hat{x}, \tilde{y}^k) \\
&=\frac{1}{2} (x^k-\hat{x})^\top \nabla_{xx}^2 f(\hat{x}, \hat{y})(x^k-\hat{x}) + (x^k-\hat{x})^\top \nabla_{xy}^2 f(\hat{x}, \hat{y})(\tilde{y}^k-\hat{y})  \\
&~~~+ o\left(\norm{x^k-\hat{x}}^2 + \norm{\tilde{y}^k-\hat{y}}^2 \right) - o\left(\norm{\tilde{y}^k-\hat{y}}^2 \right)\\
&\overset{(a)}{=} \frac{1}{2} (x^k-\hat{x})^\top \nabla_{xx}^2 f(\hat{x}, \hat{y})(x^k-\hat{x}) + (x^k-\hat{x})^\top \nabla_{xy}^2 f(\hat{x}, \hat{y})(\tilde{y}^k-\hat{y})+ o\left(\norm{x^k-\hat{x}}^2 \right) \\
&= t_k^2 \frac{1}{2} v^\top \nabla_{xx}^2 f(\hat{x}, \hat{y})v + o(t_k^2),
\end{align*}
where (a) follows from the fact that $\norm{\tilde{y}^k - \hat{y}} = O\left(\norm{x^k - \hat{x}}\right)$ for sufficiently large $k$.
Dividing by $t_k^2$ in both sides and letting $t_k\to 0$, we
complete the proof.
\end{proof}

The following example illustrates $\cl \left\{w: w \in \Gamma_1^\circ(\hat{x},y'), \forall y'\in \mathbb{B}(\hat{y},\delta) \right\}$ for some $\delta>0$.

\begin{example}
Let $n=m=1$, $X = Y = [-1,1]$.
Consider
\begin{equation*}
\min_{x\in [-1,1]} \max_{y\in [-1,1]}~ f(x,y):= -x^4 + 4x^2y^2 -y^4.
\end{equation*}
We have
\begin{equation*}
\varphi(x)= \max_{y\in[-1,1]} (-x^4 + 4x^2y^2 -y^4) =
\begin{cases}
3 x^4, & x\in \left[-\frac{\sqrt{2}}{2}, \frac{\sqrt{2}}{2}\right]~(y^*=\pm \sqrt{2}x);\\
-x^4 + 4x^2 - 1, & [-1,1]\big\backslash \left[-\frac{\sqrt{2}}{2}, \frac{\sqrt{2}}{2}\right]~(y^*=1),
\end{cases}
\end{equation*}
which is not a convex function over $[-1,1]$.
Moreover, it can be examined that $(0,0)$ is a global minimax point. In fact, it is also a local minimax point. Let $\tau(\delta)=2\delta^2$ and $\delta_0=\frac{\sqrt{2}}{2}$. Then, for any $\delta\in (0,\delta_0]$ and any $(x,y)\in [-1,1]^2$ satisfying $\abs{x}\leq \delta$ and $\abs{y}\leq \delta$, we have
$$-y^4=f(0,y) \leq f(0,0) \leq \max_{y'\in\{y\in Y: \abs{y}\leq \tau(\delta)\}}f(x,y')=3x^4.$$

Therefore, for any $\delta\in(0,1]$,
\begin{align*}
\cl \left\{w: w \in \Gamma_1^\circ(0,y'), \forall y'\in \mathbb{B}(0,\delta) \right\}
&= \cl \left(\bigcap_{y'\in \mathbb{B}(0,\delta)}\{w_1\in \mathcal{T}_{[-1,1]}^\circ(0): w_1\bot \nabla_x f(0,y')\} \right)= \mathbb{R}.
\end{align*}
Similarly, we have
$$\cl\Gamma_2^\circ(0,0)=\left\{w_2\in \mathcal{T}_{[-1,1]}^\circ(0): w_2\bot \nabla_y f(0,0) \right\} = \mathbb{R}.$$

In this case, the second-order optimality  condition \eqref{gs6} means
$\nabla_{xx}^2 f(0,0)\ge 0 $ and $ \nabla_{yy}^2 f(0,0)\le 0$.
\end{example}

In Theorem \ref{Th1}, the first-order and second-order optimality necessary conditions are given in a sense of geometry. With specific structures of sets $X$ and $Y$, we can derive the corresponding Karush-Kuhn-Tucker (KKT) systems. Specially, we assume that $X$ and $Y$ are polyhedral defined as follows:
\begin{equation}
\label{gs3}
X=\{x\in\mathbb{R}^n:Ax\leq b\}~\text{and}~Y=\{y\in\mathbb{R}^m:Cy\leq d\},
\end{equation}
where $A\in \mathbb{R}^{p\times n}$, $b\in \mathbb{R}^{p}$, $C\in \mathbb{R}^{q\times m}$ and $d\in \mathbb{R}^{q}$.

The following proposition establishes the relationship between tangent/normal cones and algebra systems when $X$ and $Y$ are defined in \eqref{gs3}.

\begin{proposition}[\cite{FP2007finite}]\label{Prop6}
Let $X$ and $Y$ be defined in \eqref{gs3}. Then,

\begin{align*}
\mathcal{T}_X(x) &= \left\{ \lambda \in \mathbb{R}^n: -A_i^\top \lambda \geq 0 ~\text{for}~ \forall i\in \mathcal{A}_X(x) \right\},\\
\mathcal{T}_Y(y) &= \left\{ \mu \in \mathbb{R}^m: -C_j^\top \mu \geq 0 ~\text{for}~ \forall j\in \mathcal{A}_Y(y) \right\},\\
\mathcal{N}_X(x) &= \left\{ -\sum_{i=1}^p \alpha_i A_i: \alpha \in \mathcal{N}_{\mathbb{R}_+^p}(b - Ax)\right\},\\
\mathcal{N}_Y(y) &= \left\{ -\sum_{j=1}^q \beta_j C_j: \beta \in \mathcal{N}_{\mathbb{R}_+^q}(d - Cy)\right\},
\end{align*}
where $A_i$ is the $i$th row vector of matrix $A$ and $C_j$ is the $j$th row vector of matrix $C$ respectively for $i=1,\cdots,p$ and $j=1,\cdots,q$, and $\mathcal{A}_X(x)$ and $\mathcal{A}_Y(y)$ are active sets of $X$ at $x$ and $Y$ at $y$, respectively.
\end{proposition}

\begin{theorem}\label{Th3}
Let the tuple $(\hat{x},\hat{y})\in X\times Y$ be a local minimax point of problem \eqref{minmax} with $X$ and $Y$ being defined in \eqref{gs3}. Then there exist multipliers $\alpha\in \mathbb{R}^p$ and $\beta\in\mathbb{R}^q$ such that
\begin{equation}
\label{FKKT}
\begin{cases}
\nabla_x f(\hat{x},\hat{y}) - \sum_{i=1}^p \alpha_i A_i=0,\\
- \nabla_y f(\hat{x},\hat{y}) - \sum_{j=1}^q \beta_j C_j=0,\\
\alpha \in \mathcal{N}_{\mathbb{R}_+^p}(b - A\hat{x}),\\
\beta \in \mathcal{N}_{\mathbb{R}_+^q}(d - C\hat{y}).
\end{cases}
\end{equation}

If, moreover, $f$ is twice continuously differentiable, we have, for any $\delta>0$, that
\begin{equation}
\label{SKKT}
\begin{cases}
\inp{v, \nabla_{xx}^2 f(\hat{x},\hat{y}) v} \geq 0 ~\text{for all}\\
~~~~~v\in \left\{ \lambda \in \mathcal{T}_X(\hat{x}): \exists\delta>0, ~ \lambda^\top \nabla_x f(\hat{x},y')=0 ~\text{for}~ y'\in \mathbb{B}(\hat{y},\delta) \right\},\\
\inp{w, \nabla_{yy}^2 f(\hat{x},\hat{y})w} \leq 0
~\text{for all}~w\in \left\{ \mu \in \mathcal{T}_Y(\hat{y}): \mu^\top \nabla_y f(\hat{x},\hat{y})=0 \right\}.
\end{cases}
\end{equation}
\end{theorem}

The proof of Theorem \ref{Th3} directly follows from Lemma \ref{Lem1}, Theorem \ref{Th1} and Proposition \ref{Prop6}. Thus, we neglect it here.

We call \eqref{FKKT} the first-order KKT system of problem \eqref{minmax} (see also \cite[Proposition 1.2.1]{FP2007finite} for the KKT system for variational inequalities with polyhedral sets) and \eqref{FKKT}-\eqref{SKKT} the second-order KKT system of problem \eqref{minmax}.

\begin{definition}
We call that $(\hat{x},\hat{y})\in X\times Y$ is a first-order stationary point of problem \eqref{minmax} if it satisfies \eqref{gs2-1}-\eqref{gs2-2}. Moreover, if $(\hat{x},\hat{y})$ also satisfies \eqref{gs6-1}-\eqref{gs6-2}, we call it a second-order stationary point of problem \eqref{minmax}.
\end{definition}

The existence results of the first-order stationary points can be obtained by using existing results in \cite[Proposition 2.2.3, Corollary 2.2.5]{FP2007finite}.
Let $F(x,y)=\begin{pmatrix}
\nabla_x f(x,y)\\
- \nabla_y f(x,y)
\end{pmatrix}$.
\begin{enumerate}
\item[(i)] If there exits a bounded open set $\mathcal{Z}\subseteq X\times Y$ and a point $(\bar{x},\bar{y})\in (X\times Y)\cap \mathcal{Z}$ such that
    $$\inp{F(x,y),\begin{pmatrix}x-\bar{x}\\y-\bar{y}\end{pmatrix}}\geq 0,~ \forall (x,y)\in (X\times Y)\cap \bd(\mathcal{Z}),$$
    then problem \eqref{minmax} has at least a first-order stationary point.

\item[(ii)] Specially, if $X$ and $Y$ are bounded, the first-order stationary point set of problem \eqref{minmax} is nonempty.
\end{enumerate}

We know from \cite[Proposition 21]{JNJ2020local} that a global minimax point can be neither a local minimax point nor a stationary point. However, some global minimax points can be the first-order stationary points.

The following proposition claims that under mild conditions a class of global minimax points are first-order stationary points.

\begin{proposition}
\label{Prop2}
Let $f$ be continuously differentiable over $X\times Y$, and $(\hat{x},\hat{y})$ be a global minimax point of \eqref{minmax} satisfying
\begin{equation*}
\hat{y} \in \limsup_{x\to \hat{x}} \left(\argmax_{y'\in Y} f(x,y') \right),
\end{equation*}
where ``\,$\limsup$'' denotes outer limit (\cite[Definition 4.1]{RW2009variational}), then $(\hat{x},\hat{y})$ is a first-order stationary point.
\end{proposition}

\begin{proof}
Since $(\hat{x},\hat{y})$ is a global minimax point, we have
\begin{equation}\label{Gminimax}
f(\hat{x},y) \overset{(a)}{\leq} f(\hat{x},\hat{y})  \overset{(b)}{\leq} \max_{y'\in Y}f(x,y').
\end{equation}
The inequality (a) of \eqref{Gminimax} implies \eqref{gs2-2}. In the sequel, we only consider \eqref{gs2-1} through inequality (b) of \eqref{Gminimax}.
Since
$$\hat{y} \in \limsup_{x\to \hat{x}} \left(\argmax_{y'\in Y} f(x,y') \right),$$
without loss of generalization, we assume that there exists a sequence $\tilde{y}^k\in \mathcal{Y}^k$ such that $\tilde{y}^k \to  \hat{y}$.
By a similar procedure to the proof for (i) of Theorem \ref{Th1}, we have
\begin{align*}
0 &\leq \nabla_x f(\hat{x}, \tilde{y}^k)^\top (x^k-\hat{x})  + o\left(\norm{x^k-\hat{x}}\right)\\
&= \nabla_x f(\hat{x}, \hat{y})^\top (x^k-\hat{x}) + (\nabla_x f(\hat{x}, \tilde{y}^k) - \nabla_x f(\hat{x}, \hat{y}) )^\top (x^k-\hat{x})  + o\left(\norm{x^k-\hat{x}}\right)\\
&= \nabla_x f(\hat{x}, \hat{y})^\top (x^k-\hat{x})   + o\left(\norm{x^k-\hat{x}}\right),
\end{align*}
which implies that $-\nabla_x f(\hat{x}, \hat{y}) \in \mathcal{N}_X(\hat{x})$.
\end{proof}

In general, a global minimax point can be neither a local minimax point nor a stationary point \cite[Proposition 21]{JNJ2020local}. We give the following example to show this.

\begin{example}[{\cite[Figure 2]{JNJ2020local}}]\label{Eg2}
Let $n=m=1$, $X = [-1,1]$ and $Y = [-5,5]$. Consider the following minimax problem
\begin{equation}
\label{eg-gs1}
\min_{x\in[-1,1]}\max_{y\in [-5,5]}~f(x,y):=  xy-\cos(y).
\end{equation}

By direct calculation, we have
\begin{align*}
\varphi(x) &= \max_{y\in[-5,5]} (xy-\cos(y)) \\
&=
\begin{cases}
x\cdot( \pi -  \arcsin(-x)) - \cos(\pi-\arcsin(-x)), & x\in[0,1];\\
x\cdot( -\pi -  \arcsin(-x)) - \cos(-\pi-\arcsin(-x)), & x\in[-1,0],
\end{cases}
\end{align*}
where the optima is achieved when $y= \pi - \arcsin(-x)$ and $y= -\pi - \arcsin(-x)$, respectively. It can observe from the definition of $\varphi(x)$ that $x=0$ is the minimizer. In this case, $(0,-\pi)$ and $(0,\pi)$ are two global minimax points. However, they both fail to satisfy \eqref{gs2-1}-\eqref{gs2-2}, that is,
\begin{equation*}
  \begin{cases}
  0\in  x + \sin(y) + \mathcal{N}_{[-1,1]}(x),\\
  0\in y + \mathcal{N}_{[-5,5]}(y),
  \end{cases}
\end{equation*}
which has a unique solution $(0,0)$. Thus, neither $(0,-\pi)$ nor $(0,\pi)$ is a first-order stationary point, which implies from Theorem \ref{Th1} that they cannot be local minimax points either. Therefore, a global minimax point can be neither a local minimax point nor a first-order stationary point.

Next, we show that even $(0,0)$ is not a local minimax point. For any $y$ satisfying $0<\abs{y}\leq \delta$ with any sufficiently small $\delta>0$, we have
$$-\cos(y)=f(0,y) > f(0,0)=-1,$$
which, according to the definition of local minimax points in Definition \ref{Def4}, concludes that $(0,0)$ is not a local minimax point. Therefore, problem \eqref{eg-gs1} here does not have a local minimax point even both $X$ and $Y$ are bounded.
\end{example}

Sometimes we can find that a global minimax point may be a stationary point (Example \ref{Eg1}). In the following proposition, we conclude some sufficient conditions such that a global minimax point is a local minimax point.

\begin{proposition}
\label{Prop3}
Let $(\hat{x},\hat{y})$ be a global minimax point and $f$ be Lipschitz continuous over $X\times Y$. If $\max_{y'\in Y} f(x,y')$ has a unique solution for all $x$ in a neighborhood of $\hat{x}$, then $(\hat{x},\hat{y})$ is a local minimax point.
\end{proposition}

\begin{proof}
Since $\max_{y'\in Y} f(x,y')$ has a unique solution for all $x$ in a neighborhood of $\hat{x}$, we use $\bar{y}(x)$ to denote this unique solution. By \cite[Lemma 4.1]{X2010uniform}, we know that
$$\norm{\bar{y}(x) - \hat{y}} \to 0 ~\text{as}~x\to \hat{x},$$
which implies that there exists a $\delta_0>0$ such that for any $x\in X$ satisfying $\norm{x-\hat{x}} \leq \delta\leq \delta_0$, $\tau(\delta)\to 0$ where $\tau(\delta):=\sup_{\{x\in X:\norm{x-\hat{x}} \leq \delta\}}\norm{\bar{y}(x) - \hat{y}}$. As $(\hat{x},\hat{y})$ is a global minimax point, we have for any $x\in X$ and $y\in Y$ that
\begin{equation*}
f(\hat{x},y) \leq f(\hat{x},\hat{y})  \leq \max_{y'\in Y}f(x,y').
\end{equation*}
This indicates that for $x$ satisfying $\norm{x-\hat{x}} \leq \delta (\leq \delta_0)$ and $y$ satisfying $\norm{y - \hat{y}} \leq \tau(\delta)$, we have
\begin{equation*}
f(\hat{x},y) \leq f(\hat{x},\hat{y})  \leq \max_{y'\in Y}f(x,y')= f(x,\bar{y}(x)) = \max_{y'\in \{y\in Y: \norm{y - \hat{y}} \leq \tau(\delta)\}}f(x,y').
\end{equation*}
Thus, $(\hat{x},\hat{y})$ is a local minimax point based on Definition \ref{Def4}.
\end{proof}

Obviously, when $f(x,\cdot)$ is strictly concave for all $x$ in a neighborhood of $\hat{x}$, the  condition for the uniqueness of the maximization problem holds.

To end this section, we summarize relationships between saddle points, local saddle points, global minimax points, local minimax points and first-order and second-order stationary points in Figure \ref{fig-3}.
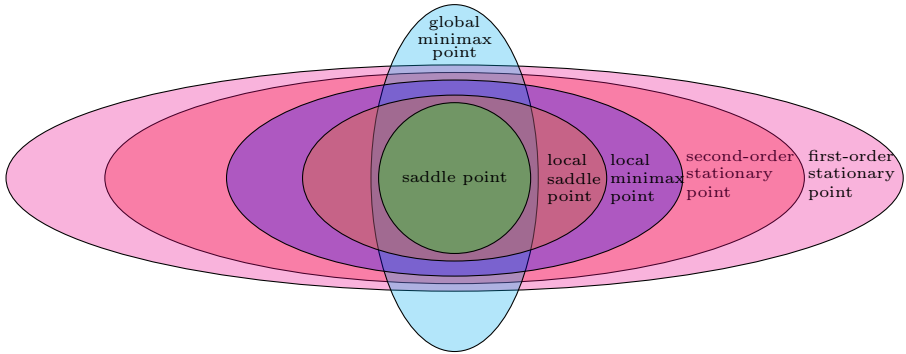
\begin{figure}[htp]
\begin{center}
\begin{tikzpicture}
\draw[fill=red,fill opacity=0.3] (0,0) ellipse (4.6cm and 1.4cm);
\draw (0.1,0.3)node[right=80pt,black,fill opacity=1]{\tiny{second-order}};
\draw (0.1,0.05)node[right=80pt,black,fill opacity=1]{\tiny{stationary}};
\draw (0.1,-0.2)node[right=80pt,black,fill opacity=1]{\tiny{point}};

\draw[fill=magenta,fill opacity=0.3] (0,0) ellipse (5.9cm and 1.5cm);
\draw (1.35,0.3)node[right=90pt,black,fill opacity=1]{\tiny{first-order}};
\draw (1.35,0.05)node[right=90pt,black,fill opacity=1]{\tiny{stationary}};
\draw (1.35,-0.2)node[right=90pt,black,fill opacity=1]{\tiny{point}};

\draw[fill=cyan,fill opacity=0.3] (0,0) ellipse (1.1cm and 2.3cm);
\draw (0,2.05) node[black,fill opacity=1]{\tiny{global}};
\draw (0,1.85) node[black,fill opacity=1]{\tiny{minimax}};
\draw (0,1.65) node[black,fill opacity=1]{\tiny{point}};

\draw[fill=blue,fill opacity=0.3] (0,0) ellipse (3cm and 1.3cm);
\draw (-0.2,0.25) node[right=60pt,black,fill opacity=1]{\tiny{local}};
\draw (-0.2,0) node[right=60pt,black,fill opacity=1]{\tiny{minimax}};
\draw (-0.2,-0.25) node[right=60pt,black,fill opacity=1]{\tiny{point}};

\draw[fill=orange,fill opacity=0.3] (0,0) ellipse (2cm and 1.1cm);
\draw (-0.5,0.25) node[right=45pt,black,fill opacity=1]{\tiny{local}};
\draw (-0.5,0) node[right=45pt,black,fill opacity=1]{\tiny{saddle}};
\draw (-0.5,-0.25) node[right=45pt,black,fill opacity=1]{\tiny{point}};

\draw[fill=green,fill opacity=0.3] (0,0) circle (1cm) node[black,fill opacity=1]{\tiny{saddle point}};
\end{tikzpicture}
\end{center}
\caption{Venn diagram for saddle points, minimax points and stationary points} \label{fig-3}
\end{figure}

\section{Generative adversarial networks}\label{Sec4}

In this section, we consider the first-order and second -order optimality conditions of the GAN using nonsmooth activation functions, which can be formulated as nonsmooth nonconvex-nonconcave  min-max problem (\ref{minmax}).

The GAN is one of the most popular generative models in machine learning. It is comprised of two ingredients: the generator, which creates samples that are intended to follow the same distribution as the training data, and the discriminator, which examines samples to determine whether they are real or fake. For more motivations and advantages of GANs, one can refer to \cite{G2014generative}. Recently, Wang gave a mathematical introduction to GANs in \cite{W2020mathematical}.

The plain vanilla GAN model can be formulated as \eqref{GAN}, where $D$ and $G$ are given by feedforward neural networks with parameters $x$ and $y$, respectively. Theoretically, the universal approximation theorem tells us that, under certain mild conditions, a feedforward neural network can approximate any Borel measurable function. Moreover, the derivatives of the feedforward neural network can also approximate the derivatives of this function arbitrarily well. The activation function is a function from $\mathbb{R}$ to $\mathbb{R}$ that is used to compute the hidden layer values and introduce the nonlinear property. There are several commonly-used activation functions, such as ReLU $\sigma(z)=\max\{0,z\}$, the logistic sigmoid $\sigma(z) = 1/(1+ \exp(-z))$, the softplus activation function $\sigma(z) = \ln(1+ \exp(z))$, etc.

We give an intuition for $D$ and $G$ which are consist of linear models with activation functions in the following example.

\begin{example}\label{Eg4}
Consider that the discriminator $D$ is a single-layer network with a logistic sigmoid activation function \cite{GLWM2020limiting} and the generator $G$ is a two-layer network with an activation function  $\sigma$ as follows
\begin{align*}
G(x,\xi_2):= W_2 \bm{\sigma} (W_1\xi_2+b_1) + b_2,   \quad \quad
D(y,\xi_1):= \frac{1}{1+ \exp(y^\top\xi_1 )},
\end{align*}
where $x=(\mathrm{vec}(W_1)^\top, \mathrm{vec}(W_2)^\top,b_1^\top,b_2^\top)^\top$ and $\mathrm{vec}(\cdot)$ denotes the columnwise vectorization operator of matrices, $W_1\in \mathbb{R}^{s \times s_2}$, $b_1\in \mathbb{R}^{s }$, $W_2\in \mathbb{R}^{s_1\times s}$, $b_2\in \mathbb{R}^{s_1}$ and $\bm{\sigma}:\mathbb{R}^s\to \mathbb{R}^s$.
  Here, $\bm{\sigma}$ is a separable vector activation function that aggregates the individual neuron activations in the layer.

In this case, the GAN model \eqref{GAN} can be explicitly written as
\begin{equation}
\label{GAN-2}
\begin{aligned}
\min_{x \in X} \max_{y \in Y}~  &f(x,y)=\Bigg( \mathbb{E}_{P_{1}}\left[ \log\left( \frac{1}{1+ \exp(y^\top \xi_1)} \right) \right]  \\
&+  \mathbb{E}_{P_2}\left[\log\left( 1 - \frac{1}{1+ \exp(y^\top (W_2 \bm{\sigma} (W_1\xi_2+b_1) + b_2) )} \right)\right] \Bigg).
\end{aligned}
\end{equation}
If $X$ and $Y$ are compact and $\bm{\sigma}$ is continuous, by Proposition \ref{Prop1}, problem (\ref{GAN-2}) has a global minimax point.
\end{example}

Obviously, if $D(\cdot,\xi_1)$ and $G(\cdot,\xi_2)$ are smooth (i.e. $\bm{\sigma}$ is smooth), the necessary optimality conditions in Theorem \ref{Th1} hold. Next, we focus on the nonsmooth case with the ReLU activation function.

\begin{proposition}\label{Prop8}
Let $f$ be defined in (\ref{GAN-2})  with $\bm{\sigma}(\cdot)=(\cdot)_+$. Assume that support sets $\Xi_1$ and $\Xi_2$ are bounded. Then the following statements hold.
\begin{enumerate}
\item[(i)] $f$ is locally Lipschitz continuous and twice semidifferentiable in $X\times Y$.
\item[(ii)] In addition, if $\xi_2$ is a continuous random variable, $\Xi_2$ has a nonempty interior and $\|(W_1)_i\|+|(b_1)_i|>0$,  $i=1,\cdots,s$, then $f$ is twice continuously differentiable at $(x,y)$.
\end{enumerate}
\end{proposition}

\begin{proof}
Let
\begin{align*}
\rho_1(y)=\mathbb{E}_{P_{1}}\left[ \log\left( D(y, \xi_1) \right) \right],~ \rho_2(x,y)= \mathbb{E}_{P_2}\left[\log\left( 1 - D(y, G(x,\xi_2))\right)\right].
\end{align*}
(i) Since for any fixed  $\xi_2\in \Xi_2$, $G(x,\xi_2)$ and
$\log\left( 1 - \frac{1}{1+ \exp(y^\top G(x,\xi_2))} \right)$ are locally Lipschitz continuous in $X\times Y$,  the local Lipschitz continuity of $f(x,y)=\rho_1(y) + \rho_2( x,y)$ follows the continuous differentiability of \emph{log} and \emph{exp} functions. Moreover, the twice semidifferentiability follows directly from Example \ref{Eg3}.

(ii) Obviously, $f$ is twice continuously differentiable with respect to $y$. In what follows, we only consider the continuous differentiability of $f$ with respect to $x$. Since $\Xi_2$ has a nonempty interior and $\xi_2$ is a continuous random variable, we know that $P_2(\{\xi_2: (W_1\xi_2+b_1)_i=0\}) = 0$. Thus, $\Xi_0:=\cup_{i=1}^s\{\xi_2: (W_1\xi_2+b_1)_i=0\}$ is $P_2$-zero. Note that $\log\left( 1 - D(y, G(x,\xi_2))\right)$ is continuously differentiable with respect to $x$ for every $\xi_2\in \Xi_2\backslash \Xi_0$. Since $P_2(\Xi_0)=0$, we have
\begin{align*}
\rho_2(x,y) = \int_{\Xi_2\backslash \Xi_0 } \log\left( 1 - D(y, G(x,\xi_2))\right) P_2(\d\xi_2).
\end{align*}
According to \cite[Theorem 7.57]{SDR2014lectures}, we have $\int_{\Xi_2\backslash \Xi_0 } \log\left( 1 - D(y, G(x,\xi_2))\right) P_2(\d\xi_2)$ is continuously differentiable. Thus, $\rho_2(x,y)$ is continuously differentiable and
\begin{align*}
\nabla_x \rho_2(x,y) &= \nabla_x \int_{\Xi_2\backslash \Xi_0 } \log\left( 1 - D(y, G(x,\xi_2))\right) P_2(\d\xi_2)\\
&= \int_{\Xi_2\backslash \Xi_0 } \nabla_x  \log\left( 1 - D(y, G(x,\xi_2))\right) P_2(\d\xi_2).
\end{align*}
By using \cite[Theorem 7.57]{SDR2014lectures} again, we have $\nabla_x\rho_2(x,y)$ is continuously differentiable at $(x,y)$. Thus, $f$ is twice continuously differentiable at $(x,y)$.
\end{proof}

By directly using Proposition \ref{Prop8}, we can apply Theorems \ref{Th5} and \ref{Th1} to problem \eqref{GAN-2}.

\begin{proposition}
Suppose the assumptions of Proposition \ref{Prop8} hold. Let $(\hat{x},\hat{y})$ be a local minimax point of problem \eqref{GAN-2}.  Then the first-order necessary optimality conditions \eqref{NonS1st-1}-\eqref{NonS1st-2} hold at $(\hat{x},\hat{y})$. If, in addition, $\xi_2$ is a continuous random variable, $\Xi_2$ has a nonempty interior and $\|(W_1)_i\|+|(b_1)_i|>0$,  $i=1,\cdots,s$, then the first-order and second-order necessary optimality conditions  \eqref{gs2-1}-\eqref{gs2-2} and \eqref{gs6-1}-\eqref{gs6-2} hold at $(\hat{x},\hat{y})$.
\end{proposition}

In the remainder of this section, we show how to verify whether a pair $(x,y)\in X\times Y$ is a first-order or second-order stationary point of problem \eqref{GAN-2} in Example \ref{Eg4} by Proposition \ref{Prop8}. To this end, we assume  $\|(W_1)_i\|+|(b_1)_i|>0$,  $i=1,\cdots,s$, $\xi_1, \xi_2$ are  continuous random variables with bounded support sets, and $\Xi_2$ has a nonempty interior.
Let $X=[a, b]$ and $Y=[c, d]$ where $a, b \in \mathbb{R}^n,
c, d\in\mathbb{R}^m$, $a < b$, and $c<d$ with $n=(s+1)(s_1+s_2)$ and $m=s_1$.

Since $\|(W_1)_i\|+|(b_1)_i|>0$,  $i=1,\cdots,s$, we can generate independent identically distributed (iid) samples $\{\xi_1^j\}_{j=1}^N$ of $\xi_1$ and $\{\xi_2^j\}_{j=1}^N$ of $\xi_2$ such that $(W_1\xi_2^j+b_1)_i\neq 0$ for $i=1,\cdots,s$ and $j=1,\cdots,N$.
We consider the following min-max problem
\begin{equation}
\label{SAA-GAN-2}
\begin{aligned}
\min_{x\in X}\max_{y\in Y}~\hat{f}_N(x,y):= &\frac{1}{N}\sum_{i=1}^N \Bigg( \log\left( \frac{1}{1+ \exp(y^\top \xi_1^i)} \right) \\
&+  \log\left( 1 - \frac{1}{1+ \exp(y^\top (W_2 (W_1\xi_2^i+b_1)_+ + b_2) )} \right) \Bigg).
\end{aligned}
\end{equation}

Use the existing automatic differentiation technique, such as back-propagation, we can compute
$\nabla_x \hat{f}_N(x,y)$,  $\nabla_y \hat{f}_N(x,y)$, $\nabla_{xx}^2 \hat{f}_N(x,y)$,  $\nabla_{yy}^2 \hat{f}_N(x,y)$. Moreover,
we have
\begin{equation*}
\mathcal{T}_X(x) = \mathcal{T}_X^\circ(x)= \left\{ v \in \mathbb{R}^n \, :\,
v_i \in\begin{cases}
[0, \infty), & \text{if}~x_i = a_i\\
(-\infty, \infty), & \text{if}~a_i < x_i <b_i\\
(-\infty, 0], & \text{if}~x_i = b_i
\end{cases}\right\},
\end{equation*}
\begin{equation*}
\mathcal{T}_Y(y) = \mathcal{T}_Y^\circ(y)= \left\{ w \in \mathbb{R}^m \, :
\,  w_j \in \begin{cases}
[0, \infty), & \text{if}~y_j = c_j\\
(-\infty, \infty), & \text{if}~c_j < y_j <d_j\\
(-\infty, 0], & \text{if}~y_j = d_j
\end{cases}\right\}
\end{equation*}
and
\begin{align*}
\Gamma_1^\circ(x,y)&=\Gamma_1(x,y)=\{v\in \mathcal{T}_X(x): v\bot \nabla_x \hat{f}_N(x,y)\},\\
\Gamma_2^\circ(x,y)&=\Gamma_2(x,y)=\{w\in \mathcal{T}_Y(y): w\bot \nabla_y \hat{f}_N(x,y)\}.
\end{align*}
By Theorem \ref{Th1}, if $(\hat{x},\hat{y})$ is a local minimax point of problem (\ref{SAA-GAN-2}), then $(\hat{x},\hat{y})$ must satisfy the first-order and second-order optimality  conditions:
$$\begin{cases}(\nabla_x \hat{f}_N(\hat{x},\hat{y}))_i\ge 0,
 & \text{if}~x_i = a_i;\\
(\nabla_x \hat{f}_N(\hat{x},\hat{y}))_i=0, & \text{if}~a_i < x_i <b_i;\\
(\nabla_x \hat{f}_N(\hat{x},\hat{y}))_i\le 0, & \text{if}~x_i = b_i
\end{cases}
~\text{and}~
\begin{cases}(\nabla_y \hat{f}_N(\hat{x},\hat{y}))_j\le 0,
 & \text{if}~y_j = c_j;\\
(\nabla_y \hat{f}_N(\hat{x},\hat{y}))_j=0, & \text{if}~c_j < y_j <d_j;\\
(\nabla_y \hat{f}_N(\hat{x},\hat{y}))_j\ge 0, & \text{if}~y_j = d_j
\end{cases}
$$
for $i=1,\ldots, n$, $j=1,\ldots, m$,  and
\begin{align*}
&\inp{v, \nabla_{xx}^2 \hat{f}_N(\hat{x},\hat{y}) v} \geq 0~\text{for all}~v\in \left\{\bar{v}: \exists \delta>0, \bar{v} \in \Gamma_1(\hat{x},y'), \forall y'\in \mathbb{B}(\hat{y},\delta) \right\},\\
&\inp{w, \nabla_{yy}^2 \hat{f}_N(\hat{x},\hat{y}) w} \leq 0~\text{for all}~w\in \Gamma_2(\hat{x},\hat{y}).
\end{align*}

The following proposition tells that the above procedures can ensure an exponential rate of convergence with respect to sample size $N$.

\begin{proposition}Suppose the assumptions of Proposition \ref{Prop8} hold.
 If $(x_N, y_N)$ is a first-order (or second-order stationary point) of problem \eqref{SAA-GAN-2} with iid samples $ \{\xi_1^j\}_{j=1}^N$ and $\{\xi_2^j\}_{j=1}^N$ of $\xi_1$ and $\xi_2$ respectively,  then $(x_N, y_N)$ converges to a first-order (or second-order stationary point) of problem \eqref{GAN-2} exponentially with respect to $N$.
\end{proposition}

\begin{proof}
Denote
\begin{align*}
h(z)=
\begin{pmatrix}
\nabla_x f(x,y)\\
-\nabla_y f(x,y)
\end{pmatrix}, \quad \quad
H(z)=
\begin{pmatrix}
\sup_{v\in \mathcal{V}(x,y)}\inp{v, -\nabla_{xx}^2 f(x,y) v}\\
\sup_{w\in \mathcal{W}(x,y)}\inp{w, \nabla_{yy}^2 f(x,y) w}
\end{pmatrix},\\
\hat{h}_N(z)=
\begin{pmatrix}
\nabla_x \hat{f}_N(x,y)\\
-\nabla_y \hat{f}_N(x,y)
\end{pmatrix}, \quad \quad
\widehat{H}_N(z)=
\begin{pmatrix}
\sup_{v\in \mathcal{V}(x,y)}\inp{v, -\nabla_{xx}^2 \hat{f}_N(x,y) v}\\
\sup_{w\in \mathcal{W}(x,y)}\inp{w, \nabla_{yy}^2 \hat{f}_N(x,y) w}
\end{pmatrix},
\end{align*}
where $z=(x^\top,y^\top)^\top$, $\mathcal{V}(x,y):=\mathbb{B}(0,1) \cap  \cup_{\delta>0} \cl \left\{\bar{v}: \exists \delta>0, \bar{v} \in \Gamma_1^\circ(x,y'), \forall y'\in \mathbb{B}(y,\delta) \right\}$ and $\mathcal{W}(x,y):=\mathbb{B}(0,1) \cap \cl\Gamma_2^\circ(x,y)$.

According to the twice continuous differentiability of $f$ (see Proposition \ref{Prop8}) and the boundedness of $\Xi_1$ and $\Xi_2$, we have $\hat{h}_N(z) \to h(z)$ and $\widehat{H}_N(z) \to H(z)$ exponentially fast uniformly in any compact subset of $\mathcal{Z}\subseteq Z:=X\times Y$ (\cite[Theorem 7.73]{SDR2014lectures}). That is, for any given $\epsilon >0$, there exist $C=C(\epsilon)$ and $\beta=\beta(\epsilon)$, such that
\begin{equation*}
\prob\left\{ \sup_{z\in \mathcal{Z}}\norm{\hat{h}_N(z) - h(z)} \geq \epsilon \right\} \leq C e^{-N\beta} ~\text{and}~ \prob\left\{ \sup_{z\in \mathcal{Z}}\abs{\widehat{H}_N(z)- H(z)} \geq \epsilon \right\} \leq C e^{-N\beta}.
\end{equation*}
Without loss of generality, we assume that $z_N=(x_N^\top, y_N^\top)^\top\in \mathcal{Z}$. Denote the following general growth functions:
\begin{align*}
\psi_1(\tau)&:=\inf\{\d(0, h(z) + \mathcal{N}_Z(z)): z\in \mathcal{Z},\,  \d(z, \mathcal{S}_1)\geq \tau \},\\
\psi_2(\tau)&:=\inf\{\norm{(H(z))_+}: z\in \mathcal{Z}, \, \d(z, \mathcal{S}_2)\geq \tau \},
\end{align*}
where $\mathcal{S}_1$ and $\mathcal{S}_2$ are the sets satisfying \eqref{gs2-1}-\eqref{gs2-2} and \eqref{gs6-1}-\eqref{gs6-2}, respectively, and ``$\d$'' denotes the distance from a point to a set. Let the related functions
\begin{equation*}
\Psi_1(t):=\psi_1^{-1}(t) + t ~\text{and}~ \Psi_2(t):=\psi_2^{-1}(t) + t,
\end{equation*}
where $\psi_i^{-1}(t):=\sup\{\tau: \psi_i(\tau) \leq \eta\}$ for $i=1,2$, which satisfy $\Psi_i(t) \to 0$ as $t\downarrow 0$ for $i=1,2$.

Then, by a conventional discussion (see e.g. \cite{CSS2019convergence}), we have
\begin{equation*}
\d(z_N, \mathcal{S}_1) \leq  \Psi_1\left( \sup_{z\in Z}\norm{\hat{h}_N(z) - h(z)} \right)~\text{and}~\d(z_N, \mathcal{S}_2) \leq  \Psi_2\left( \sup_{z\in Z}\abs{\widehat{H}_N(z)- H(z)} \right).
\end{equation*}
Thus, we have
\begin{equation*}
\prob\left\{ \d(z_N, \mathcal{S}_1) \geq \Psi_1(\epsilon) \right\} \leq C e^{-N\beta} ~\text{and}~ \prob\left\{ \d(z_N, \mathcal{S}_2) \geq \Psi_2(\epsilon) \right\} \leq C e^{-N\beta},
\end{equation*}
which shows that $z_N$ converges to a first-order stationary point in $\mathcal{S}_1$ (or a first-order stationary point in $\mathcal{S}_2$)  exponentially with respect to $N$.
\end{proof}

\section{Conclusions}\label{Sec5}

Many nonconvex-nonconcave min-max problems in dada sciences do not have saddle points. In this paper, we provide sufficient conditions for the existence of global and local minimax points of constrained nonsmooth nonconvex-nonconcave min-max problem (\ref{minmax}). Moreover, we give the first-order and second-order optimality conditions of local minimax points of problem (\ref{minmax}), and use these conditions to define the first-order and second-order stationary points of (\ref{minmax}). The relationships between saddle points, local saddle points, global minimax points, local minimax points, stationary points are summarized in Figure \ref{fig-3}. Several examples are employed to illustrate our theoretical results. To demonstrate applications of these optimality conditions in deep learning, we propose a method to verify the optimality conditions at any given point of generative adversarial network (\ref{GAN-2}).


\end{document}